\pdfoutput=1
\documentclass[12pt, oneside]{amsart}
\usepackage{amsmath,amssymb}
\usepackage{longtable}
\usepackage{booktabs,array}
\usepackage{microtype}
\usepackage{placeins}
\usepackage{tikz}
\usepackage{etoolbox}
\usepackage[textwidth=6in, textheight=9in,marginpar=0.75in]{geometry}
\usepackage[colorlinks=true,linkcolor=blue,citecolor=blue,urlcolor=blue]{hyperref}

\newtheorem{theorem}{Theorem}[section]
\newtheorem{proposition}[theorem]{Proposition}
\newtheorem{lemma}[theorem]{Lemma}
\newtheorem{corollary}[theorem]{Corollary}

\theoremstyle{definition}

\newtheorem{remark}[theorem]{Remark}

\newcommand{\bbz}{{\mathbb Z}}
\newcommand{\bbq}{{\mathbb Q}}
\newcommand{\dbc}{\Sigma}

\newcommand{\kn}[2]{\textup{#1#2}}
\newcommand{\para}[1]{\par\medskip\noindent\textbf{#1}\par\nobreak\smallskip\noindent\ignorespaces}
\newcommand{\coderepo}{\url{https://github.com/Pinocchio315/negative-definite-plumbings}}
\newcommand{\nNewUtwo}{41}

\newcommand{\nAdopted}{43}
\newcommand{\nObstrD}{35}
\newcommand{\nObstrHF}{6}
\newcommand{\nBadOne}{28}
\newcommand{\nBadTwo}{15}
\newcommand{\nNegDefBadOne}{147}
\newcommand{\nNegDefBadTwo}{33}
\newcommand{\nIdentified}{43}
\newcommand{\nUthreeAll}{66}

\newcommand{\nUthreeNonalt}{47}
\newcommand{\nOpenUonetwo}{437}
\newcommand{\nOpenUtwothree}{713}
\newcommand{\nOpenAll}{1{,}150}
\newcommand{\nHypCert}{490}
\newcommand{\nMixed}{149}
\newcommand{\nNoDef}{331}
\newcommand{\nNoDefRec}{299}
\newcommand{\nNegDef}{180}
\newcommand{\nGraphFour}{32}
\newcommand{\nNegDefOne}{59}
\newcommand{\nNegDefTwo}{121}
\newcommand{\nNoTest}{16}
\newcommand{\nUthreeNoCond}{57}
\newcommand{\nUthreeSigned}{9}
\newcommand{\nUthreeSigmaZero}{13}
\newcommand{\nBFTested}{2{,}264}
\newcommand{\nBFObstructed}{96}
\newcommand{\nBFKnownThree}{30}
\newcommand{\nBFNegCtrl}{9{,}673}
\newcommand{\nBFPubCtrl}{25}
\newcommand{\nRefAlt}{4{,}878}
\newcommand{\nDetCand}{13{,}346}
\newcommand{\nExclCross}{11{,}903}
\newcommand{\nExclAlex}{1{,}379}
\newcommand{\nExclCover}{57}
\newcommand{\nMinDiag}{7}
\newcommand{\nNbrChanges}{91}
\newcommand{\nNbrTypes}{8}
\newcommand{\nTwelveRef}{1{,}288}
\newcommand{\nTwelveDiag}{18}
\newcommand{\nTwelveChanges}{216}

\newcommand{\nPosVol}{639}
\newcommand{\nNoPlumb}{970}

\newcommand{\nAlgCheck}{3{,}848}
\newcommand{\nAlgTrees}{316}
\newcommand{\nAlgThm}{33}
\newcommand{\nAlgCrossMin}{14}
\newcommand{\nAlgCrossMax}{16}
\newcommand{\nThmCNoPlumb}{45}
\newcommand{\nNewAll}{112}
\newcommand{\nClassicalNew}{69}
\newcommand{\nCtrl}{469}
\newcommand{\nCtrlFourTwo}{156}
\newcommand{\nCtrlFourThree}{38}
\newcommand{\nCtrlFourThreeBoth}{33}
\newcommand{\nCtrlFourThreeExp}{4}
\newcommand{\nCtrlAlt}{350}
\newcommand{\nCtrlAltBoth}{33}
\newcommand{\nCtrlAltGoeritz}{34}
\makeatletter
\def\LT@makecaption#1#2#3{%
  \LT@mcol\LT@cols c{\hbox to\z@{\hss\parbox[t]\textwidth{\reset@font\@tempcnta=64\relax
    \@makecaption{#1{#2}}{#3}}\hss}}}
\makeatother

\makeatletter
\patchcmd{\@settitle}{\bfseries}{\normalfont\scshape}{}%
  {\ClassError{amsart}{Could not set the title in small caps}{}}
\patchcmd{\@settitle}{\uppercasenonmath\@title}{}{}%
  {\ClassError{amsart}{Could not preserve title capitalization}{}}
\patchcmd{\@setauthors}{\MakeUppercase{\authors}}{{\scshape\authors}}{}%
  {\ClassError{amsart}{Could not set the author name in small caps}{}}
\g@addto@macro\@setauthors{%
  \vspace{-1.5ex}
  \begin{center}
    \normalfont\footnotesize
    Center for Geometry and Physics, Institute for Basic Science,\\
    Pohang 37673, South Korea
  \end{center}
}
\makeatother

\hypersetup{pdftitle={Negative-Definite Plumbings and New Unknotting Numbers},pdfauthor={Seong-Jin Lee}}

\begin{document}

\title[Negative-Definite Plumbings and New Unknotting Numbers]{Negative-Definite Plumbings\\
  and New Unknotting Numbers}
\author[Seong-Jin Lee]{Seong-Jin Lee}

\begin{abstract}
We use Neumann's normal form to decide whether the double branched cover of a knot bounds a negative-definite plumbing
tree in either orientation, starting from the recognised Seifert or graph-manifold structure of the cover or from an
algebraic diagram of the knot. For many non-alternating knots with at most thirteen crossings whose unknotting numbers
were unknown, this identifies the cases in which the correction terms of the cover can be computed from such a plumbing,
and shows for all the others that neither orientation of the cover bounds one. We determine \nNewAll{} new unknotting numbers. For \nAdopted{}
of them, the lower bound rests on Heegaard Floer homology computed from the plumbing; for the other \nClassicalNew{}, it
comes from linking pairings and from a theorem of McCoy and Park on special alternating knots. Finally, we show that the
knot \kn{13n}{1587} is a further explicit counterexample to the Bernhard--Jablan conjecture, besides \kn{13n}{3370}.
\end{abstract}

\maketitle

\section{Introduction}\label{sec:intro}

The unknotting number $u(K)$ of a knot $K$ is the minimal number of crossing changes needed to turn a diagram of $K$
into a diagram of the unknot, where the minimum is taken over all diagrams of $K$. Although it is easy to define, it is
hard to compute: for many knots with at most thirteen crossings only a range of values is known \cite{KnotInfo,Lee}.
Upper bounds are obtained from explicit crossing changes, and lower bounds from knot invariants. For small values of
$u(K)$, lower bounds can be obtained from the double branched cover $\dbc(K)$. By the Montesinos trick, if $K$ can be
unknotted by $n$ crossing changes, then $\dbc(K)$ is obtained by surgery on an $n$-component link in $S^3$ with
half-integral framings \cite{Montesinos}, and the correction terms of Heegaard Floer homology obstruct the existence of
such surgery descriptions \cite{OS05,Owens}. Applying these obstructions requires the Floer homology of $\dbc(K)$.

In \cite{Lee}, the correction terms of $\dbc(K)$ were obtained from the Goeritz matrix for alternating knots
\cite{OS05}, from Greene's spanning tree model when $\dbc(K)$ is an $L$-space \cite{Greene}, and from
negative-definite star-shaped plumbings for Montesinos knots, whose double branched covers are Seifert fibred
\cite{OSplumbed,Nemethi}. These computations left the unknotting numbers of many non-alternating knots with at most
thirteen crossings undetermined. For these knots we start from the cover itself. We recognise $\dbc(K)$ with Regina
\cite{Regina}, and when it is a Seifert fibred space or a graph manifold, possibly with several Seifert pieces, we
recover a plumbing tree from its Seifert invariants and gluing data; no Montesinos presentation of the knot is needed.
When Regina does not recognise the cover, we read a plumbing tree from an algebraic diagram of the knot instead
(Lemma~\ref{lem:algebraic}), using the description of the double branched covers of arborescent knots by Bonahon and
Siebenmann \cite{BS}.

Our approach rests on Neumann's normal form and his theorem on the positive index of plumbing trees
\cite{Neumann81,Neumann89}. For an oriented graph manifold that is a rational homology sphere and whose normal form is
a tree, the normal form is negative definite exactly when the manifold bounds a negative-definite plumbing tree
(Corollary~\ref{cor:negdef}). We construct candidate normal forms from the recognised data and identify their plumbed
manifolds with $\pm\dbc(K)$ by triangulation, or we obtain the normal forms from an algebraic diagram. When the normal
forms of both $\dbc(K)$ and $-\dbc(K)$ are found and
neither is negative definite, this proves that no orientation of $\dbc(K)$ bounds a negative-definite plumbing tree.
The search for candidates is bounded, but the conclusions rest only on the identified normal forms and on Neumann's
theorem. When a negative-definite normal form exists, its lattice gives the correction terms of $\dbc(K)$ if it has at
most two bad vertices \cite{OSplumbed}, and the reduced Floer homology if it has at most one \cite{OSplumbed,Nemethi}.

Let $\mathcal S$ be the set of the \nOpenAll{} non-alternating knots with at most thirteen crossings whose unknotting
number lies in $[1,2]$ (\nOpenUonetwo{} knots) or in $[2,3]$ or $[2,4]$ (\nOpenUtwothree{} knots) according to
\cite{Lee}. Our first result describes for which of these knots the plumbing calculation is available.

\begin{proposition}\label{prop:classes}
The set $\mathcal S$ is the disjoint union of subsets $\mathcal P$ and $\mathcal N$ with $|\mathcal P|=\nNegDef$ and
$|\mathcal N|=\nNoPlumb$, such that the following hold. For every $K\in\mathcal P$, the normal form of $\dbc(K)$ or of
$-\dbc(K)$ is negative definite and has at most two bad vertices. For every $K\in\mathcal N$, neither $\dbc(K)$ nor
$-\dbc(K)$ bounds a negative-definite plumbing tree.
\end{proposition}

For \nPosVol{} knots of $\mathcal N$, the cover $\dbc(K)$ has positive simplicial volume: it is hyperbolic for
\nHypCert{} knots, and toroidal with a hyperbolic piece in its JSJ decomposition for \nMixed{} knots; no orientation of
such a cover bounds any plumbing tree. For the other \nNoDef{} knots of $\mathcal N$, the cover is a graph manifold, and
the intersection matrices of the normal forms of both orientations have positive eigenvalues. For \nGraphFour{} of them,
Regina does not recognise the cover, which has four Seifert pieces; these knots are arborescent, and we obtain the
normal forms from algebraic diagrams (Section~\ref{sec:geometry} and Table~\ref{tab:map}). Thus the lattice formula of
Ozsv\'ath and Szab\'o applies to every knot of $\mathcal P$ and to no knot of $\mathcal N$; for the latter, the
correction terms of $\dbc(K)$ must be computed by other means.

The second result determines new unknotting numbers. The algebraic unknotting number $u_a(K)$ is the least number of
crossing changes that turn $K$ into a knot with Alexander polynomial one, and the four-dimensional clasp number
$c_4(K)$ is the least number of double points of a normally immersed disc in $B^4$ bounded by $K$; both are at most
$u(K)$.

\begin{theorem}\label{thm:main}
\begin{enumerate}\renewcommand{\theenumi}{\roman{enumi}}\renewcommand{\labelenumi}{(\theenumi)}
\item\label{it:plumbing} The knots of Table~\ref{tab:new-u2} have unknotting number $2$, and those of
Table~\ref{tab:new-u3} have unknotting number $3$.
\item\label{it:rank2} The \nUthreeAll{} knots of Table~\ref{tab:u3-66} satisfy $u(K)=u_a(K)=3$.
\item\label{it:u5} The special alternating knots \kn{13a}{4770}, \kn{13a}{4787} and \kn{13a}{4801} satisfy
$u(K)=c_4(K)=5$.
\end{enumerate}
\end{theorem}

Statement~\eqref{it:plumbing} comes from the plumbing calculation for knots of $\mathcal P$. For \nObstrD{} of the
\nNewUtwo{} knots of Table~\ref{tab:new-u2}, the correction terms of the normal form exclude $u(K)=1$ through the
formula of Ni and Wu \cite{NiWu}; for the other \nObstrHF{}, the reduced Floer homology excludes it through the mapping
cone formula \cite{Gainullin}. For the two knots of Table~\ref{tab:new-u3}, the obstruction of Owens \cite{Owens} and
the criterion of Traczyk \cite{Traczyk} exclude $u(K)=2$.

Statements~\eqref{it:rank2} and~\eqref{it:u5} complement the plumbing calculation. A form of the rank-two linking
pairing obstruction of Borodzik and Friedl \cite{BF} needs no plumbing; by Proposition~\ref{prop:classes}, neither
$\dbc(K)$ nor $-\dbc(K)$ bounds a negative-definite plumbing tree for \nThmCNoPlumb{} of the \nUthreeNonalt{}
non-alternating knots of Table~\ref{tab:u3-66}. It also applies to the \nUthreeSigmaZero{} knots of
Table~\ref{tab:u3-66} with $\sigma(K)=0$, for which the correction-term obstructions to $u=2$ used here are not
available. For special alternating knots, a theorem of McCoy and Park \cite{McCoyPark} reduces the lower bound to a
finite computation in a single diagram. Theorem~\ref{thm:main} determines \nNewAll{} new unknotting numbers. None of
them is recorded in KnotInfo (version 2026.10.5), where the unknotting numbers of all these knots are still given as
ranges \cite{KnotInfo}.

The third result concerns the conjecture of Bernhard and Jablan \cite{Bernhard,Jablan}. For a nontrivial knot $K$, let
$N_{\min}(K)$ be the set of knots obtained by changing one crossing in a minimal diagram of $K$, and let
\[
    u^s_{\mathrm{BJ}}(K)=1+\min\{u(K')\mid K'\in N_{\min}(K)\}.
\]
Then $u(K)\le u^s_{\mathrm{BJ}}(K)$ \cite{BH1705}, and the conjecture asserts that equality holds for every
knot.

\begin{theorem}\label{thm:bj}
Every crossing change in every minimal diagram of \kn{13n}{1587} yields a knot of unknotting number at least $2$.
Consequently, $u^s_{\mathrm{BJ}}(\kn{13n}{1587})=3>2=u(\kn{13n}{1587})$.
\end{theorem}

For \kn{13n}{1587}, Brittenham and Hermiller showed that the unknotting number is smaller than the weak
Bernhard--Jablan number, in which every crossing change of an unknotting sequence is made in a minimal diagram of the
current knot, but the equality $u=u^s_{\mathrm{BJ}}$ remained open \cite{BH1705}; the value $u(\kn{13n}{1587})=2$ was
determined in \cite{Lee}. Theorem~\ref{thm:bj} settles this case. Its new assertion is the lower bound for every knot
obtained by a crossing change in a minimal diagram. We prove it by finding, with the reference diagrams of Brittenham
and Hermiller \cite{BH1705}, \nMinDiag{} minimal diagrams from which all these knots arise, and bounding the unknotting
numbers of the resulting knots. Hence \kn{13n}{1587} is a further explicit counterexample to the
Bernhard--Jablan conjecture, besides \kn{13n}{3370} \cite{BH1705,Lee}.

The computations separate the search for presentations from the verification of the conclusions. Every plumbing tree
obtained from Regina's data is identified with $\pm\dbc(K)$ by triangulating its plumbed manifold, and the normal forms
read from algebraic diagrams are confirmed by a second, different diagram of each knot. The inertia of every intersection matrix
is computed exactly, and the orientations of the normal forms used for Theorem~\ref{thm:main}\eqref{it:plumbing} are
confirmed by a Casson--Walker identity. The correction-term obstructions
are recomputed by separately written programs, whose shared inputs are specified in Section~\ref{sec:verify}.

The paper is organised as follows. Section~\ref{sec:background} collects background on double branched covers,
unknotting obstructions and plumbings, and describes the computation of the normal form. Section~\ref{sec:new} proves
Proposition~\ref{prop:classes} and Theorem~\ref{thm:main}\eqref{it:plumbing}, Section~\ref{sec:66} proves
Theorem~\ref{thm:main}\eqref{it:rank2} and~\eqref{it:u5}, and Section~\ref{sec:13n1587} proves Theorem~\ref{thm:bj}.
The new values are listed in the appendix.

\section{Background}\label{sec:background}

We write $\det K=|\Delta_K(-1)|$ for the determinant of a knot $K$, $\sigma(K)$ for its signature, and $\dbc(K)$ for
the double cover of $S^3$ branched along $K$. We follow the convention of \cite{Owens}: changing a positive crossing
leaves the signature unchanged or raises it by two, and changing a negative crossing leaves it unchanged or lowers it
by two. The double branched cover $\dbc(K)$ is a rational homology sphere with $|H_1(\dbc(K))|=\det K$. Since $\det K$ is odd,
$\dbc(K)$ has a unique spin structure, which we use to identify $\mathrm{Spin}^c(\dbc(K))$ with $H_1(\dbc(K))$. For an
integral symmetric $r\times r$ matrix $A$ of odd determinant, let $l(A)$ be the linking pairing on $\bbz^r/A\bbz^r$
with values in $\bbq/\bbz$ given by $([v],[w])\mapsto v^TA^{-1}w\bmod\bbz$. If $V$ is a Seifert matrix of $K$, then $\sigma(K)=\operatorname{sign}(V+V^T)$, and the linking
pairing $l(K)$ of $\dbc(K)$ is isometric to $l(V+V^T)$ \cite{BF}.

\subsection{Double branched covers and unknotting}\label{sec:bg-dbc}

By the Montesinos trick, crossing changes on $K$ give surgery descriptions of $\dbc(K)$ \cite{Montesinos}. In
particular, if $u(K)=1$, then $\dbc(K)\cong S^3_{\pm D/2}(C)$ for a knot $C\subset S^3$, where $D=\det K$
\cite{OS05}. For two or more crossing changes we use the following refinement, which records their signs.

\begin{theorem}[Owens \cite{Owens}]\label{thm:montesinos}
Suppose that $K$ can be unknotted by changing $p$ positive and $n$ negative crossings, where $n=\sigma(K)/2$. Then
$\dbc(K)$ is obtained by surgery on a $(p+n)$-component link in $S^3$ with linking matrix $\frac12Q$, where $Q$ is a
positive-definite integral matrix with $Q\equiv I\pmod2$ and $\det Q=\det K$, exactly $n$ of whose diagonal entries
are congruent to $3$ modulo $4$.
\end{theorem}

We use $\mathbb F_2$ coefficients. For a rational homology sphere $Y$ and $\mathfrak t\in\mathrm{Spin}^c(Y)$, the
correction term $d(Y,\mathfrak t)\in\bbq$ is the minimal grading of a nonzero homogeneous element in the image of the
map $HF^\infty(Y,\mathfrak t)\to HF^+(Y,\mathfrak t)$ \cite{OS03}, and the reduced group
$HF_{\mathrm{red}}(Y,\mathfrak t)$ is the cokernel of this map. Reversing the orientation changes the sign of every
correction term \cite{OS03} and, by duality, does not change the dimensions of the reduced groups.

If $u(K)=1$, then $H_1(\dbc(K))$ is cyclic, generated by a class of self-linking $\pm2/D$ \cite{Lickorish,BF}, and after reversing the orientation if necessary $\dbc(K)\cong S^3_{D/2}(C)$. In the labelling of $\mathrm{Spin}^c(S^3_{D/2}(C))$ by $\bbz/D$
used by Ni and Wu, their formula \cite{NiWu} gives
\begin{equation}\label{eq:niwu}
    d(S^3_{D/2}(C),i)=d(L(D,2),i)-2\max\{V_{\lfloor i/2\rfloor},V_{\lfloor(D+1-i)/2\rfloor}\},\qquad 0\le i<D,
\end{equation}
where the integers $V_j\ge0$ depend only on $C$, satisfy $V_j\ge V_{j+1}\ge V_j-1$ \cite{NiWu2}, and vanish
for $j$ at least the genus of $C$ \cite{NiWu}. Thus, if $u(K)=1$, the correction terms of $\dbc(K)$ or of
$-\dbc(K)$ have the form \eqref{eq:niwu} under an affine identification of $\mathrm{Spin}^c$ structures with $\bbz/D$
whose linear part is compatible with the linking pairing; compare \cite{OwensStrle}.

The mapping cone formula gives a further condition. For $0\le i<D$ and $n\in\bbz$ put $s_n=\lfloor(i+nD)/2\rfloor$,
and let $n_i\in\{0,-1\}$ be an index with $|s_{n_i}|=\min\{\lfloor i/2\rfloor,\lfloor(D+1-i)/2\rfloor\}$, the index
of the maximum in \eqref{eq:niwu}. By \cite[Cor.~14 and Prop.~15]{Gainullin}, together with $H_s=V_{-s}$,
$HF_{\mathrm{red}}(S^3_{D/2}(C),i)$ is the direct sum of a part of dimension
\[
    T_i=\sum_{n\ne n_i}V_{|s_n|},
\]
which comes from the towers of the mapping cone, and of $\bigoplus_{n\in\bbz}Q_{s_n}$ for finite-dimensional vector
spaces $Q_s$ that depend only on $C$. Each $Q_s$ occurs for exactly two labels, $2s$ and $2s+1$ modulo $D$, so
\begin{equation}\label{eq:cone}
    \operatorname{rank}HF_{\mathrm{red}}(S^3_{D/2}(C),i)-T_i=c_i+c_{i-1},\qquad c_i\in\bbz_{\ge0},\quad i\in\bbz/D,
\end{equation}
with $c_i$ the sum of $\dim Q_s$ over $2s\equiv i\pmod D$. Since $D$ is odd, \eqref{eq:cone} has the unique solution
$c_i=\frac12\sum_{j=0}^{D-1}(-1)^jr_{i-j}$, where $r_i$ denotes its left-hand side, and the condition is that these
numbers are non-negative integers. The correction terms determine $V_j$ for
$j\le\max_i\min\{\lfloor i/2\rfloor,\lfloor(D+1-i)/2\rfloor\}$; we apply \eqref{eq:cone} only when the last of these
$V_j$ is zero, so that all $V_j$, and hence all $T_i$, are known. Lidman used this argument to show that
$u(\kn{11n}{102})=2$ \cite{Lidman}.

Suppose that $u(K)=2$, and replace $K$ by its mirror if necessary so that $\sigma(K)\ge0$. Since the signature changes
by at most two at each crossing change and a positive change cannot lower it, the two changes are negative if
$\sigma(K)=4$, and if $\sigma(K)=2$ they are either one positive and one negative, $(p,n)=(1,1)$, or two negative,
$(p,n)=(0,2)$. In the cases $n=\sigma(K)/2$, Theorem~\ref{thm:montesinos} provides a matrix
\[
    Q=\begin{pmatrix}2m_1-1&2a\\2a&2m_2-1\end{pmatrix},\qquad \det Q=\det K,
\]
with exactly $n$ of $m_1,m_2$ even. Replacing the half-integral surgery by an integral one gives a positive-definite
four-manifold bounded by $\dbc(K)$ whose intersection form is, after reordering the basis,
$\widetilde Q=\left(\begin{smallmatrix}M&I\\ I&2I\end{smallmatrix}\right)$ with $2M-I=Q$ \cite{Owens}. For a
positive-definite integral matrix $A$ of rank $r$ and odd determinant, let
$m_A(g)=\min\{(\xi^TA^{-1}\xi-r)/4\}$ over the characteristic covectors $\xi$ ($\xi_i\equiv A_{ii}\pmod 2$) in the
class $g\in\bbz^r/A\bbz^r$. Ozsv\'ath and Szab\'o's inequality, in the form of \cite{Owens}, then
gives an isomorphism $\phi$ from $\bbz^4/\widetilde Q\bbz^4$ to $\mathrm{Spin}^c(\dbc(K))$ with
$m_{\widetilde Q}(g)\ge d(\dbc(K),\phi(g))$ and $m_{\widetilde Q}(g)\equiv d(\dbc(K),\phi(g))\pmod 2$ for every $g$. If
no matrix $Q$ admits such an isomorphism, the pattern is excluded. The remaining pattern $(0,2)$ for $\sigma(K)=2$ is
controlled by the Jones polynomial. By Lickorish and Millett, $V_K(e^{i\pi/3})=\varepsilon_K(i\sqrt3)^{d_K}$ with
$\varepsilon_K=\pm1$, where $d_K$ is the nullity of $V+V^T$ modulo $3$, that is $d_K=\dim H_1(\dbc(K);\mathbb F_3)$
\cite{LM}. Traczyk's criterion relates $\varepsilon_K$ to the signs of $d_K$ crossing changes that unknot $K$
\cite{Traczyk} (see also \cite{ShewellBrockway}); for $d_K=2$ it says that two changes of the same
sign force $\varepsilon_K=1$. Hence, if $d_K=2$ and $\varepsilon_K=-1$, the pattern $(0,2)$ is impossible. If
$\sigma(K)=0$, the pattern $(1,1)$ is not covered by these arguments.

The algebraic unknotting number $u_a(K)$ is the least number of crossing changes that turn $K$ into a knot with
Alexander polynomial one; clearly $u_a(K)\le u(K)$. Borodzik and Friedl bound $u_a(K)$ from below in terms of the
Blanchfield pairing $\lambda(K)$.

\begin{theorem}[Borodzik--Friedl \cite{BF}]\label{thm:bf}
If $K$ can be turned into a knot with Alexander polynomial one by $k$ crossing changes, then $\lambda(K)$ is presented
by a hermitian $k\times k$ matrix $A(t)$ over $\bbz[t^{\pm1}]$ such that $A(1)$ is diagonal with entries $\pm1$,
whose signs are determined by the signs of the changed crossings. The integral symmetric matrix
$C=A(-1)$ satisfies $|\det C|=\det K$, $\operatorname{sign}C=\sigma(K)+\operatorname{sign}A(1)$ and
$l(C)\cong2\,l(K)$.
\end{theorem}

Since $A(1)$ is diagonal, the diagonal entries of $C$ are congruent modulo four to those of $A(1)$, and its
off-diagonal entries are even \cite{BF}. We use these conditions in Section~\ref{sec:66}.

A knot is \emph{special alternating} if it has an alternating diagram in which one of the checkerboard surfaces is a
Seifert surface; all crossings of such a diagram have the same sign. Let $c_4$ denote the four-dimensional clasp
number, the least number of double points of a normally immersed disc in $B^4$ bounded by the knot
\cite{OwensStrle}. McCoy and Park proved the following for non-split special alternating links; we state the case of
knots.

\begin{theorem}[McCoy--Park \cite{McCoyPark}]\label{thm:mccoy-park}
Let $K$ be a special alternating knot. The following are equivalent: \textup{(i)} $u(K)=|\sigma(K)|/2$;
\textup{(ii)} $c_4(K)=|\sigma(K)|/2$; \textup{(iii)} $K$ can be unknotted by $|\sigma(K)|/2$ crossing changes in any
alternating or minimal diagram.
\end{theorem}

\subsection{Plumbings and the Neumann normal form}\label{sec:bg-plumbing}

A \emph{plumbing graph} is a finite tree $\Gamma$ with vertex set $V$ and an integer weight $e_v$ at each vertex $v$.
To $\Gamma$ we associate a framed link in $S^3$: each vertex $v$ gives an unknot with framing $e_v$, and the unknots
of two vertices form a Hopf link, as in Figure~\ref{fig:plumbing}(a), exactly when the vertices are joined by an
edge. Surgery on this framed link gives a closed oriented $3$-manifold $Y_\Gamma$, the \emph{plumbed manifold} of
$\Gamma$. Equivalently, $Y_\Gamma$ is the boundary of the $4$-manifold $X_\Gamma$ obtained by plumbing disc bundles
over $S^2$ with Euler numbers $e_v$ according to $\Gamma$.

\begin{figure}[htbp]
\centering
\begin{tikzpicture}[vx/.style={circle,fill,inner sep=1.5pt}, lab/.style={font=\footnotesize}]
\node[vx,label={[lab]above:$e_v$}] (v) at (0,0) {};
\node[vx,label={[lab]above:$e_w$}] (w) at (1.1,0) {};
\draw (v)--(w);
\node at (1.85,0) {$\longleftrightarrow$};
\draw[thick] (3.2,0) ellipse (0.75 and 0.45);
\draw[thick] (4.3,0) ellipse (0.75 and 0.45);
\begin{scope}
  \clip (3.75,0.306) circle (0.14);
  \draw[white,line width=3.6pt] (4.3,0) ellipse (0.75 and 0.45);
  \draw[thick] (3.2,0) ellipse (0.75 and 0.45);
\end{scope}
\begin{scope}
  \clip (3.75,-0.306) circle (0.14);
  \draw[white,line width=3.6pt] (3.2,0) ellipse (0.75 and 0.45);
  \draw[thick] (4.3,0) ellipse (0.75 and 0.45);
\end{scope}
\node[lab] at (2.3,0.55) {$e_v$};
\node[lab] at (5.2,0.55) {$e_w$};
\node at (2.6,-1.65) {(a)};
\begin{scope}[xshift=7.6cm]
  \node[vx,label={[lab]above:$-2$}] (A) at (0,0) {};
  \node[vx,label={[lab]above:$-2$}] (C) at (1.15,0) {};
  \node[vx,label={[lab]above:$-2$}] (B) at (2.3,0) {};
  \node[vx,label={[lab]left:$-2$}] (a1) at (-0.85,0.6) {};
  \node[vx,label={[lab]left:$-3$}] (a2) at (-0.85,-0.6) {};
  \node[vx,label={[lab]right:$-3$}] (b1) at (3.15,0.6) {};
  \node[vx,label={[lab]below:$-2$}] (b2) at (3.15,-0.6) {};
  \node[vx,label={[lab]below:$-2$}] (b3) at (4.0,-0.6) {};
  \draw (a1)--(A)--(a2) (A)--(C)--(B) (b1)--(B)--(b2)--(b3);
  \draw[dashed,rounded corners=4pt] (-1.85,-1.2) rectangle (0.35,1.0);
  \draw[dashed,rounded corners=4pt] (1.95,-1.2) rectangle (4.5,1.0);
  \node[lab] at (1.15,-0.35) {chain};
  \node at (1.3,-1.65) {(b)};
\end{scope}
\end{tikzpicture}
\caption{(a) An edge of a plumbing graph and the corresponding framed unknots, which form a Hopf link. (b) The normal
form of $-\dbc(\kn{13n}{611})$ (Table~\ref{tab:new-u2}). Each dashed box contains a node and its legs and corresponds
to a Seifert piece; the regular fibres of the two pieces meet twice on the JSJ torus, which gives the chain $[2]$, a
single vertex of weight $-2$. The tree is negative definite with $\det Q_\Gamma=9=\det\kn{13n}{611}$, and its two
nodes are its bad vertices.}
\label{fig:plumbing}
\end{figure}
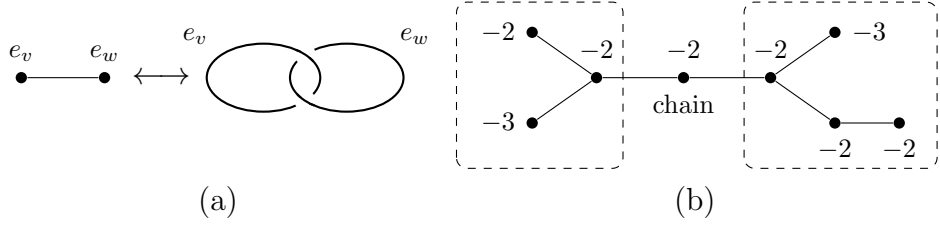

The \emph{intersection matrix} $Q_\Gamma$ of $\Gamma$ is the symmetric $V\times V$ matrix with $(Q_\Gamma)_{vv}=e_v$,
$(Q_\Gamma)_{vw}=1$ if $v$ and $w$ are adjacent, and $(Q_\Gamma)_{vw}=0$ otherwise. It is the linking matrix of the
framed link above, and it represents the intersection form of $X_\Gamma$ in the basis of $H_2(X_\Gamma;\bbz)$ given by
the zero sections of the disc bundles. If $\det Q_\Gamma\ne0$, then $Y_\Gamma$ is a rational homology sphere and
$H_1(Y_\Gamma)\cong\operatorname{coker}Q_\Gamma$. We call $\Gamma$ \emph{negative definite} if $Q_\Gamma$ is. We write
$\delta_v$ for the valency of $v$. A \emph{node} is a vertex of valency at least three, and $v$ is a \emph{bad}
vertex if $e_v>-\delta_v$ \cite{OSplumbed}.

Let $\Gamma$ be star-shaped, with central weight $e$ and $m$ legs with weights $-a_{i1},\dots,-a_{ik_i}$
($a_{ij}\ge2$, read from the centre outwards), and put
\[
    \frac{\alpha_i}{\beta_i}=[a_{i1},\dots,a_{ik_i}]=a_{i1}-\cfrac{1}{a_{i2}-\cfrac{1}{\ddots-\cfrac{1}{a_{ik_i}}}}.
\]
Then $Y_\Gamma$ is the Seifert fibred space over $S^2$ with Seifert invariants
$(0;-e-m;(\alpha_1,\alpha_1-\beta_1),\dots,(\alpha_m,\alpha_m-\beta_m))$ in the notation of
\cite{Neumann81}; in the notation of Regina \cite{Regina} it is
$\mathrm{SFS}\,[S^2:(\alpha_1,\beta_1)\cdots(\alpha_m,\beta_m)\,(1,e)]$, as we checked by triangulating such plumbed
manifolds. Regina absorbs the term $(1,e)$ into the last exceptional fibre, which it writes as
$(\alpha_m,\beta_m+e\alpha_m)$; for instance, $\mathrm{SFS}\,[S^2:(3,1)(3,1)(3,2)(5,-3)]$ in Section~\ref{sec:u2} stands
for $\mathrm{SFS}\,[S^2:(3,1)(3,1)(3,2)(5,2)\,(1,-1)]$. A \emph{graph manifold} is a closed three-manifold whose JSJ pieces are Seifert fibred \cite{Waldhausen}.
Every graph manifold is the boundary of a plumbing, and one that is a rational homology sphere is $Y_\Gamma$ for a
plumbing tree $\Gamma$ \cite{Neumann81}. If $\Gamma$ is in normal form (see below) and is not a
chain, the tori corresponding to the chains between nodes cut $Y_\Gamma$ into Seifert fibred pieces, one for each
node, and form a reduced graph structure in the sense of Waldhausen \cite{Neumann81}. If the chain joining
two nodes has weights $-b_1,\dots,-b_r$ and $[b_1,\dots,b_r]=p/q$ in lowest terms, the regular fibres of the two
pieces meet $p$ times on their common torus \cite{Neumann81}; the chain is empty if and only if
$p=1$. For example, Regina identifies $\dbc(\kn{13n}{611})$ as
$\mathrm{SFS}\,[D:(2,1)(3,2)]\cup\mathrm{SFS}\,[D:(3,1)(3,2)]$. The intersection matrix of the normal form of
$\dbc(\kn{13n}{611})$ has two positive eigenvalues, while the normal form of $-\dbc(\kn{13n}{611})$, shown in
Figure~\ref{fig:plumbing}(b), is negative definite; each has one node for each of the two pieces. Reversing the orientation of a piece replaces every fibre
$(\alpha,\beta)$ by $(\alpha,\alpha-\beta)$. Accordingly, the legs $[2]$ and $[3]$ of the left node in
Figure~\ref{fig:plumbing}(b) give the fibres $(2,1)$ and $(3,1)$ of the first piece with the reversed orientation,
and the legs $[3]$ and $[2,2]$ of the right node give the fibres $(3,1)$ and $(3,2)$ of the second piece, which the
reversal interchanges.

Different plumbing trees can have homeomorphic plumbed manifolds, and Neumann's plumbing calculus describes exactly
when this happens \cite{Neumann81}. We use the following moves, which we call \emph{Neumann moves}
(Figure~\ref{fig:moves}), and their inverses.
\begin{enumerate}
\item[(a)] \emph{Blowing down a vertex of valency two.} A vertex of weight $\varepsilon=\pm1$ with two neighbours is
deleted, its neighbours are joined by an edge, and their weights are changed by $-\varepsilon$.
\item[(b)] \emph{Blowing down a vertex of valency one.} A vertex of weight $\varepsilon=\pm1$ with one neighbour is
deleted, and the weight of the neighbour is changed by $-\varepsilon$.
\item[(c)] \emph{$0$-chain absorption.} A vertex of weight $0$ and valency two is deleted together with its two
neighbours, which are replaced by a single vertex whose weight is the sum of their weights and which is joined to all
their other neighbours.
\end{enumerate}
Neumann moves do not change the oriented homeomorphism type of the plumbed manifold \cite{Neumann81}. Up to a change of
basis, a blow-up or blow-down with $\varepsilon=-1$ adds or removes an orthogonal summand $(-1)$ of the intersection
matrix, so it preserves negative definiteness; each of the other moves changes the number of positive eigenvalues
\cite{Neumann89}.

\begin{figure}[htbp]
\centering
\begin{tikzpicture}[vx/.style={circle,fill,inner sep=1.7pt}, lab/.style={font=\normalsize}, label distance=2pt,
  eq/.style={font=\Large}, panel/.style={font=\normalsize}]
\def\fanL#1{\draw (#1)--++(-0.62,0.2) (#1)--++(-0.68,0) (#1)--++(-0.62,-0.2);}
\def\fanR#1{\draw (#1)--++(0.62,0.2) (#1)--++(0.68,0) (#1)--++(0.62,-0.2);}
\begin{scope}
\node[vx,label={[lab]above:$e_1+\varepsilon$}] (a1) at (0,0) {};
\node[vx,label={[lab]above:$\varepsilon$}] (a2) at (1.2,0) {};
\node[vx,label={[lab]above:$e_2+\varepsilon$}] (a3) at (2.4,0) {};
\draw (a1)--(a2)--(a3); \fanL{a1} \fanR{a3}
\node[eq] at (1.2,-0.82) {$\wr$};
\node[vx,label={[lab]above:$e_1$}] (a4) at (0.55,-1.8) {};
\node[vx,label={[lab]above:$e_2$}] (a5) at (1.85,-1.8) {};
\draw (a4)--(a5); \fanL{a4} \fanR{a5}
\node[panel] at (1.2,-2.75) {(a)};
\end{scope}
\begin{scope}[xshift=5.3cm]
\node[vx,label={[lab]above:$e+\varepsilon$}] (b1) at (0,0) {};
\node[vx,label={[lab]above:$\varepsilon$}] (b2) at (1.2,0) {};
\draw (b1)--(b2); \fanL{b1}
\node[eq] at (0.45,-0.82) {$\wr$};
\node[vx,label={[lab]above:$e$}] (b3) at (0.45,-1.8) {};
\fanL{b3}
\node[panel] at (0.45,-2.75) {(b)};
\end{scope}
\begin{scope}[xshift=8.8cm]
\node[vx,label={[lab]above:$e_1$}] (c1) at (0,0) {};
\node[vx,label={[lab]above:$0$}] (c2) at (1.2,0) {};
\node[vx,label={[lab]above:$e_2$}] (c3) at (2.4,0) {};
\draw (c1)--(c2)--(c3); \fanL{c1} \fanR{c3}
\node[eq] at (1.2,-0.82) {$\wr$};
\node[vx,label={[lab,label distance=5pt]above:$e_1+e_2$}] (c4) at (1.2,-1.8) {};
\fanL{c4} \fanR{c4}
\node[panel] at (1.2,-2.75) {(c)};
\end{scope}
\end{tikzpicture}
\caption{Neumann moves on plumbing trees, where $\varepsilon=\pm1$ and the fans of edges stand for the rest of the
tree: blowing down a vertex of valency two (a) or one (b), and $0$-chain absorption (c). In each case the two trees
have orientation-preserving homeomorphic plumbed manifolds.}
\label{fig:moves}
\end{figure}
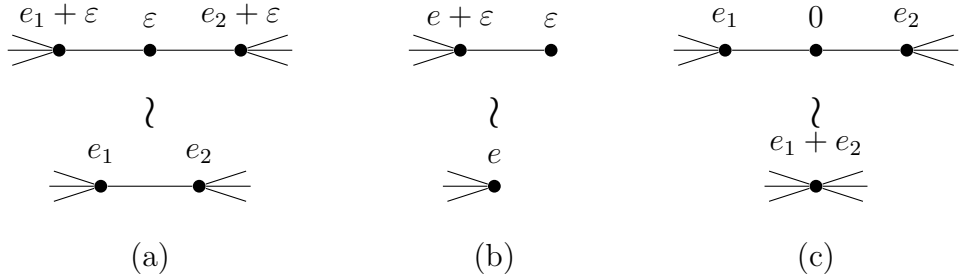

We use the normal form of \cite{Neumann89}: a plumbing tree is in \emph{normal form} if every vertex of valency at
most two has weight at most $-2$. If $Y$ bounds a plumbing tree, then it bounds a unique plumbing forest in normal
form, the \emph{normal form} of $Y$ \cite[Thm.~1.2]{Neumann89}; for the graph manifolds in this paper it is a tree,
as we find below.
The normal form of $-Y_\Gamma$ is in general not obtained from that of $Y_\Gamma$ by negating the weights; we compute
it from the Seifert invariants of $-Y_\Gamma$.

\begin{theorem}[Neumann {\cite[Thm.~1.2]{Neumann89}}]\label{thm:neumann}
Let $\Gamma_0$ be the normal form of an oriented graph manifold $Y$. If $\Gamma$ is a plumbing tree with
$Y_\Gamma\cong Y$, then $Q_\Gamma$ has at least as many positive eigenvalues as $Q_{\Gamma_0}$.
\end{theorem}

\begin{corollary}\label{cor:negdef}
Let $Y$ be an oriented graph manifold that is a rational homology sphere and whose normal form is a tree. Then $Y$
bounds a negative-definite plumbing tree if and only if its normal form is negative definite. In that case the normal
form is the negative-definite plumbing tree bounded by $Y$ with the fewest vertices \cite{Neumann89}.
\end{corollary}

Suppose that $Q_\Gamma$ is negative definite. A covector $c\in\bbz^V$ is \emph{characteristic} if
$c_v\equiv e_v\pmod2$ for every $v$; we write $c^2=c^TQ_\Gamma^{-1}c$. Restriction identifies the classes of
characteristic covectors modulo $2Q_\Gamma\bbz^V$ with $\mathrm{Spin}^c(Y_\Gamma)$; we write $[c]$ for the class of $c$.

\begin{theorem}[Ozsv\'ath--Szab\'o \cite{OSplumbed}]\label{thm:lattice}
Let $\Gamma$ be a negative-definite plumbing tree with at most two bad vertices. Then
\[
    d(Y_\Gamma,\mathfrak t)=\max\Bigl\{\frac{c^2+|V|}{4}\Bigm| c\ \text{characteristic},\ [c]=\mathfrak t\Bigr\}
    \qquad\text{for every }\mathfrak t\in\mathrm{Spin}^c(Y_\Gamma).
\]
\end{theorem}

If $\Gamma$ has at most one bad vertex, then $HF^+(-Y_\Gamma)$ is isomorphic to the combinatorial module
$\mathbb H^+(\Gamma)$ of Ozsv\'ath and Szab\'o \cite{OSplumbed}, which is determined by the graded root of
$\Gamma$ \cite{Nemethi}; the almost rational graphs of N\'emethi include those with at most one bad
vertex \cite{Nemethi}.

Negative-definite plumbing trees are also the natural setting for the $q$-series invariants $\hat Z$ of $3$-manifolds
introduced by Gukov, Pei, Putrov and Vafa \cite{GPV,GPPV}. For such a tree $\Gamma$, the invariants of $Y_\Gamma$ are
given by an explicit formula in terms of $Q_\Gamma$ \cite{GPPV}, which is unchanged by Neumann moves between weakly
negative-definite plumbing trees \cite{GM}. The negative-definite normal forms of Section~\ref{sec:geometry} are
therefore canonical inputs for $\hat Z$ of the corresponding double branched covers.

\para{Computation of the normal form of $\dbc(K)$}
We construct $\dbc(K)$ as the cyclic double cover of the orbifold obtained by $(2,0)$-filling the exterior of $K$ in
SnapPy \cite{SnapPy}, and recognise a triangulation of it with Regina \cite{Regina}, repeating with random
retriangulations and, if necessary, a search of the Pachner graph. When Regina identifies $\dbc(K)$ as a Seifert
fibred space or as a graph manifold with at most three Seifert pieces, it returns the Seifert invariants of the pieces
and the gluing matrices. From these data we form \emph{candidate trees}, with one node for each piece, as follows. We
choose an orientation of each piece. The legs at a node are the Hirzebruch--Jung expansions of the fibre invariants of
the oriented piece, and the chain joining two nodes is the expansion of $p/q$, where the regular fibres of the two
pieces meet $p$ times and $q$ is a unit modulo $p$. All node weights but one range over $[-40,40]$, and the last one is
then determined by $|\det Q_\Gamma|=\det K$, since $\det Q_\Gamma$ is an affine function of each weight; we skip the
choices for which $\det Q_\Gamma$ does not depend on the last weight, and keep last weights of absolute value at most
$120$. Every candidate is in normal form, because its nodes have valency at least three and its other weights are at
most $-2$. We accept a candidate if its plumbed manifold has a Regina name of $\dbc(K)$; we compute this name from the
tree, with conventions relating plumbing graphs to Regina's names that we fixed by triangulating plumbed manifolds.
Regina's names determine a manifold only up to orientation, so by the uniqueness of the normal form, every accepted tree
is the normal form of $\dbc(K)$ or of $-\dbc(K)$.

The search only produces candidates, and the conclusions below do not depend on its range: they rest on the accepted
trees and on the uniqueness of the normal form. For each double branched cover of Section~\ref{sec:geometry} that
Regina identifies as a graph manifold, either some accepted tree is negative definite, or two non-isomorphic weighted
trees are accepted and neither of them is negative definite; in the second case the two trees are the normal forms of $\dbc(K)$
and $-\dbc(K)$. By Corollary~\ref{cor:negdef}, this decides whether $\dbc(K)$ or $-\dbc(K)$ bounds a
negative-definite plumbing tree. The trees used in this decision, the negative-definite one in the first case and both
in the second, were also identified with $\pm\dbc(K)$ by triangulating their plumbed manifolds as surgery on the
plumbing link and recognising them with Regina; this identification does not use the conventions above. Their
inertia and determinants were computed exactly. The orientation of a negative-definite tree is fixed by comparing the linking
pairing of $\operatorname{coker}Q_\Gamma$ with that of $\dbc(K)$, and checked with the Casson--Walker invariant
(Section~\ref{sec:verify}).

\para{Algebraic diagrams}
When Regina does not recognise $\dbc(K)$, we read a plumbing tree from a diagram of $K$ instead. We regard a diagram as
a $4$-valent plane graph whose vertices are tangles, at first the crossings. If two vertices are joined by two edges
that bound a face, we merge them into one vertex, whose tangle is the sum of the two tangles. A diagram is
\emph{algebraic} if repeated merging leads to a single vertex. The boundary spheres of the tangles formed along the way
are Conway spheres that cut $K$ into the crossings, the trivial tangle outside the last sphere, and pieces bounded by
three spheres through which two strands pass. Hence $K$ is arborescent in the sense of
Bonahon and Siebenmann, who note that these are the algebraic knots of Conway \cite{BS}.

We draw each tangle $T$ in a square with its four ends at the corners, and call its two closures the one joining the
two upper ends and the two lower ends, and the one joining the two left ends and the two right ends. We fix
$\eta\in\{1,-1\}$ and attach to $T$ a plumbing tree $\Gamma_T$ with a distinguished vertex $r_T$, its root, by the
following rules (Figure~\ref{fig:algebraic}).
\begin{enumerate}
\item[(a)] A crossing whose overstrand joins the lower left corner to the upper right corner gives a single vertex of
weight $\eta$. A crossing in the other position is a quarter turn of such a crossing, and (b) applies.
\item[(b)] Turning the square of $T$ by a quarter turn adds a new root of weight $0$, joined by an edge to the old one.
Turning it by a half turn does not change $\Gamma_T$.
\item[(c)] In a merge, the two squares are placed side by side, so that the two edges of the face join the right ends of
the first tangle to the left ends of the second, after turning the squares if necessary. The roots of the two
trees are then identified, and their weights are added.
\end{enumerate}
After the last merge, we turn the square by a quarter turn if necessary, so that the remaining two edges join the upper
ends and the lower ends, and we call the resulting tree $\Gamma_D$. In a PD code, the four edges at a crossing are
listed counterclockwise, starting with the incoming understrand; placed at the upper left, lower left, lower right and
upper right corners, in this order, they form a crossing as in (a).

\begin{figure}[htbp]
\centering
\begin{tikzpicture}[scale=1.12, vx/.style={circle,fill,inner sep=1.5pt}, rt/.style={circle,draw,fill=white,inner sep=1.7pt},
  lab/.style={font=\footnotesize}, sq/.style={gray}, st/.style={thick},
  box/.style={draw,fill=white,minimum size=0.62cm,inner sep=0pt}]
\begin{scope}
  \draw[sq] (0,0) rectangle (1.6,1.6);
  \draw[st] (0,0) -- (1.6,1.6);
  \draw[st] (0,1.6) -- (0.66,0.94);
  \draw[st] (0.94,0.66) -- (1.6,0);
  \node[rt,label={[lab]right:$\eta$}] at (0.8,-0.6) {};
  \node at (0.8,-2.35) {(a)};
\end{scope}
\begin{scope}[xshift=2.6cm]
  \draw[sq] (0,0) rectangle (1.6,1.6);
  \draw[st] (0,0) -- (1.6,1.6) (0,1.6) -- (1.6,0);
  \node[box] at (0.8,0.8) {\rotatebox{90}{$T$}};
  \draw (0.8,-1.05) -- (0.45,-1.65) -- (1.15,-1.65) -- cycle;
  \node[lab] at (0.8,-1.92) {$\Gamma_T$};
  \draw (0.8,-0.6) -- (0.8,-1.05);
  \node[vx] at (0.8,-1.05) {};
  \node[rt,label={[lab]right:$0$}] at (0.8,-0.6) {};
  \node at (0.8,-2.35) {(b)};
\end{scope}
\begin{scope}[xshift=5.2cm]
  \draw[sq] (0,0) rectangle (2.7,1.6);
  \draw[st] (0,1.6) -- (0.75,0.8) -- (0,0);
  \draw[st] (0.75,0.8) .. controls (1.15,1.3) and (1.55,1.3) .. (1.95,0.8);
  \draw[st] (0.75,0.8) .. controls (1.15,0.3) and (1.55,0.3) .. (1.95,0.8);
  \draw[st] (2.7,1.6) -- (1.95,0.8) -- (2.7,0);
  \node[box] at (0.75,0.8) {$T_1$};
  \node[box] at (1.95,0.8) {$T_2$};
  \draw (1.35,-0.6) -- (0.6,-1.65) -- (1.15,-1.65) -- cycle;
  \draw (1.35,-0.6) -- (1.55,-1.65) -- (2.1,-1.65) -- cycle;
  \node[lab] at (0.88,-1.92) {$\Gamma_{T_1}$};
  \node[lab] at (1.83,-1.92) {$\Gamma_{T_2}$};
  \node[rt,label={[lab]right:$e_1+e_2$}] at (1.35,-0.6) {};
  \node at (1.35,-2.35) {(c)};
\end{scope}
\begin{scope}[xshift=8.9cm]
  \draw[sq] (0,0) rectangle (1.6,1.6);
  \draw[st] (0,0) -- (1.6,1.6) (0,1.6) -- (1.6,0);
  \draw[st] (0,1.6) .. controls (-0.15,1.95) and (1.75,1.95) .. (1.6,1.6);
  \draw[st] (0,0) .. controls (-0.15,-0.35) and (1.75,-0.35) .. (1.6,0);
  \node[box] at (0.8,0.8) {$T$};
  \draw (0.8,-0.6) -- (0.45,-1.65) -- (1.15,-1.65) -- cycle;
  \node[lab] at (0.8,-1.92) {$\Gamma_D$};
  \node[vx] at (0.8,-0.6) {};
  \node at (0.8,-2.35) {(d)};
\end{scope}
\end{tikzpicture}
\caption{The trees of Lemma~\ref{lem:algebraic}; roots are drawn as open circles. (a) A crossing as in rule (a). (b) A
quarter turn of the square of $T$ adds a new root of weight $0$. (c) A merge identifies the roots of $\Gamma_{T_1}$ and
$\Gamma_{T_2}$, whose weights $e_1$ and $e_2$ are added. (d) After the closure that joins the upper ends and the lower
ends, the root of the last tree is an ordinary vertex of $\Gamma_D$.}
\label{fig:algebraic}
\end{figure}
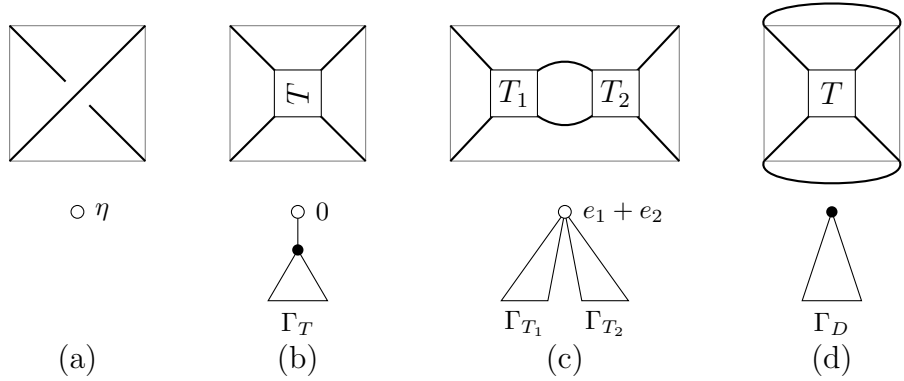

\begin{lemma}\label{lem:algebraic}
Let $D$ be an algebraic diagram of $K$. Then $Y_{\Gamma_D}\cong\pm\dbc(K)$, and the trees of $D$ for the two values of
$\eta$ differ by the sign of every weight; their plumbed manifolds are $\dbc(K)$ and $-\dbc(K)$.
\end{lemma}

\begin{proof}
By \cite[Prop.~A.1 and its proof]{BS}, the lifts of these Conway spheres cut $\dbc(K)$ into solid tori, which are the double branched covers of
the crossings and of the trivial tangle, and copies of $S^1\times P$, with $P$ a disc with two holes, which are the
double branched covers of the pieces bounded by three spheres. In such a piece, with the two strands running from left
to right, the fibres project to arcs joining the two strands or to circles, and on each boundary sphere they are lifts
of the curve that separates the left ends from the right ends, while the curve that separates the upper ends from the
lower ends lifts to a section \cite[\S A.3]{BS}; in their notation, these curves have slopes $\infty$ and $0$.

We orient $S^3$ and the plane of $D$, and give the balls of the tangles and their double branched covers the induced
orientations. For a tangle $T$, let $f_T$ and $s_T$ be lifts of the curves that separate the left ends from
the right ends and the upper ends from the lower ends. They meet once on the boundary torus of the cover, and we orient
them so that their intersection number is $1$. This fixes the pair up to a common sign, which does not matter, since
the covering involution acts as $-1$ on the homology of the boundary torus. On the boundary torus of a plumbed manifold
with a regular fibre over the root removed, we use a fibre and the meridian of the removed fibred solid torus, which is
a section, oriented in the same way. We show by induction over the merges that there is a sign $\eta_0$, depending only
on these orientations, with the following property for $\eta=\eta_0$: the double branched cover of the ball of $T$ is
orientation-preservingly homeomorphic to the plumbed manifold of $\Gamma_T$ with a regular fibre over $r_T$ removed, by
a homeomorphism that takes $f_T$ to the fibre and $s_T$ to the section.

For a crossing as in (a), the cover is a solid torus whose meridian meets $f_T$ and $s_T$ once each, so
that its class is $s_T+\eta'f_T$ up to sign, with $\eta'=\pm1$. All these crossings are the same tangle in
an oriented square, so $\eta'$ is the same for all of them. Removing a regular fibre from the plumbed manifold of a
single vertex of weight $1$ or $-1$ leaves a solid torus whose meridian has the class $s+f$ for one of the
two weights and $s-f$ for the other; this determines $\eta_0$. A quarter turn exchanges the two separating
curves and replaces $(f_T,s_T)$ by $(s_T,-f_T)$ up to a common sign, so that a slope $r$ becomes
$-1/r$. Plumbing a new root of weight $0$ has the same effect, since the plumbing exchanges fibre and section and a
vertex of weight $0$ adds no twist. A half turn preserves both separating curves, so it changes the pair only by a
common sign. In a merge, the cover of the piece between the three spheres is glued to the covers of the two tangles
fibre to fibre and section to section, so the circle bundles over the two roots and over $P$ combine into one; this
identifies the roots and adds their Euler numbers. Finally, the solid torus over the trivial tangle outside the last
sphere has as its meridian the lift of the curve that separates its two arcs, which is the section. Gluing it in fills
in the removed fibre, so $Y_{\Gamma_D}\cong\dbc(K)$ for $\eta=\eta_0$. Since the rules only add weights, changing
$\eta$ negates every weight of $\Gamma_D$, and $Y_{-\Gamma}\cong-Y_\Gamma$ for every plumbing tree $\Gamma$.
\end{proof}

We reduce $\Gamma_D$ to normal form by Neumann moves for both values of $\eta$. By the uniqueness of the normal
form, the results are the normal forms of $\dbc(K)$ and $-\dbc(K)$, whichever algebraic diagram is used. We look for
algebraic diagrams by applying random Reidemeister moves to the KnotInfo diagram and simplifying, and we decide whether
a diagram is algebraic by merging as above.

\section{New unknotting numbers}\label{sec:new}

\subsection{Double branched covers}\label{sec:geometry}

We consider the \nOpenAll{} non-alternating knots with at most thirteen crossings whose unknotting number was known
to lie in $[1,2]$ (\nOpenUonetwo{} knots) or in $[2,3]$ or $[2,4]$ (\nOpenUtwothree{} knots) \cite{Lee}; they include
the \nUthreeNonalt{} non-alternating knots of Table~\ref{tab:u3-66}. For these ranges, excluding
$u(K)=1$ or $u(K)=2$ with the obstructions of Section~\ref{sec:bg-dbc} either determines $u(K)$ or raises its lower
bound to $3$. Alternating knots are left out: for them $\dbc(K)$ is an $L$-space whose correction terms are read off
from the Goeritz matrix of a reduced alternating diagram \cite{OS05}, and they were treated in \cite{Lee}.

The obstructions require the correction terms of $\dbc(K)$, and Theorem~\ref{thm:lattice} provides them when $\dbc(K)$
or $-\dbc(K)$ bounds a negative-definite plumbing tree with at most two bad vertices. Our computation places the double
branched covers of these knots in three classes.
\begin{enumerate}
\item $\dbc(K)$ has positive simplicial volume: it is hyperbolic for \nHypCert{} knots, and toroidal with a hyperbolic
piece in its JSJ decomposition for \nMixed{} knots.
\item $\dbc(K)$ is a graph manifold, and neither $\dbc(K)$ nor $-\dbc(K)$ bounds a negative-definite plumbing tree
(\nNoDef{} knots).
\item $\dbc(K)$ or $-\dbc(K)$ bounds a negative-definite plumbing tree (\nNegDef{} knots).
\end{enumerate}
In classes~(1) and~(2) no orientation of $\dbc(K)$ bounds a negative-definite plumbing tree; for class~(2) this is a
property of the manifold rather than a failure to find a plumbing (Corollary~\ref{cor:negdef}). Hence, for all these
knots, the lattice formula of Theorem~\ref{thm:lattice} is unavailable, and the correction terms of $\dbc(K)$ must be
computed by other means. The
methods of this section apply only to the \nNegDef{} knots of class~(3). Table~\ref{tab:map} lists the classes by
previous range, together with the outcome of the tests of Sections~\ref{sec:u1} and~\ref{sec:u2} for class~(3).

\begin{table}[htbp]\centering
\caption{Double branched covers of the \nOpenAll{}
non-alternating knots of Section~\ref{sec:geometry}, by previous range \cite{Lee}. In the column $u\in[2,3],[2,4]$,
$u=2$ is excluded for \kn{13n}{2933} and \kn{13n}{2956} (Table~\ref{tab:new-u3}) and for \kn{13n}{3004} and \kn{13n}{3005}
(Table~\ref{tab:u3-66}).}\label{tab:map}
\begin{tabular}{lrr}
\toprule
$\dbc(K)$ & $u\in[1,2]$ & $u\in[2,3],[2,4]$\\
\midrule
hyperbolic (verified) & 187 & 303\\
toroidal, with a hyperbolic JSJ piece & 69 & 80\\
graph manifold, no negative-definite plumbing tree: & & \\
\quad recognised by Regina & 102 & 197\\
\quad from an algebraic diagram & 20 & 12\\
$\pm\dbc(K)$ bounds a negative-definite plumbing tree: & & \\
\quad $u=1$, resp.\ $u=2$, excluded & 41 & 4\\
\quad obstructions inconclusive & 18 & 101\\
\quad no correction-term test ($\sigma=0$) & 0 & 16\\
\midrule
total & 437 & 713\\
\bottomrule
\end{tabular}%
\end{table}

We obtained this classification as follows. We recognise $\dbc(K)$ as in Section~\ref{sec:bg-plumbing}. For the knots
whose cover is not recognised as a graph manifold, we try to certify hyperbolicity with the verified computations of
SnapPy \cite{SnapPy,HIKMOT} in Sage \cite{Sage}. The triangulation of the cover itself rarely admits a geometric
solution, so we drill a short closed geodesic and refill along its meridian, which gives another triangulation of
$\dbc(K)$, and verify a positively oriented solution of the gluing and Dehn filling equations by interval arithmetic.
This proves that $\dbc(K)$ is hyperbolic for \nHypCert{} knots. For the remaining knots we look for incompressible tori
among the vertex normal surfaces of a triangulation in quadrilateral coordinates, using Regina, and cut along them. If
one of the pieces is certified hyperbolic in the same way, then $\dbc(K)$ has positive simplicial volume: by Gromov's
additivity theorem for gluings along incompressible tori, its simplicial volume is at least that of the hyperbolic
piece, which is positive \cite{BBFIPP}. This happens for \nMixed{} knots. On the other hand, the
simplicial volume of a closed three-manifold is the sum of those of its hyperbolic pieces \cite{Soma} (see
\cite{BBFIPP}), so plumbed manifolds have simplicial volume zero. Hence in class~(1) no orientation of $\dbc(K)$
bounds a plumbing tree at all. For the graph manifolds that Regina recognises, the computation of
Section~\ref{sec:bg-plumbing} and Corollary~\ref{cor:negdef} place the knot in class~(2) or class~(3).

For the remaining \nGraphFour{} knots, an incompressible torus was found, but neither the cover nor its pieces were
recognised, and no piece was certified hyperbolic. The KnotInfo diagrams of these knots are not algebraic, since merging
stops at six vertices, but random Reidemeister moves turn each of them into an algebraic diagram with \nAlgCrossMin{}
to \nAlgCrossMax{} crossings. By Lemma~\ref{lem:algebraic}, these diagrams give the normal forms of $\dbc(K)$ and
$-\dbc(K)$. Both have four nodes, so $\dbc(K)$ is a graph manifold with four Seifert pieces, more than Regina's
recognition handles, and the intersection matrices of both have positive eigenvalues. By Corollary~\ref{cor:negdef},
these knots belong to class~(2). In the notation of Proposition~\ref{prop:classes},
$\mathcal P$ is class~(3), and $\mathcal N$ is the union of classes~(1) and~(2).

From now on, let $K$ be one of the \nNegDef{} knots of class~(3); \nNegDefOne{} of them have
$u(K)\in[1,2]$ and \nNegDefTwo{} have $u(K)\in[2,3]$ or $[2,4]$. We write $D=\det K$, $\Gamma$ for the
negative-definite normal form, and $\epsilon=\pm1$ for the sign with $Y_\Gamma\cong\epsilon\,\dbc(K)$. The tree
$\Gamma$ has one bad vertex for \nNegDefBadOne{} of these knots and two for the other \nNegDefBadTwo{}; this completes
the proof of Proposition~\ref{prop:classes}. All correction terms are therefore computed from $\Gamma$ by
Theorem~\ref{thm:lattice}. The maximum there is computed exactly: we enumerate the
characteristic covectors in an ellipsoid that contains a representative of every class, and every covector outside it
has smaller $c^2$ than these representatives.

\subsection{Unknotting number one}\label{sec:u1}
Suppose that $u(K)\in[1,2]$. If $u(K)=1$, then $H_1(\dbc(K))$ is cyclic and $\epsilon'\dbc(K)\cong S^3_{D/2}(C)$ for a
sign $\epsilon'$ and a knot $C$ (Section~\ref{sec:bg-dbc}). We therefore test, for both signs $\epsilon'$ and for every
affine identification of $\mathrm{Spin}^c(\dbc(K))$ with $\bbz/D$ that is compatible with the linking pairing, whether
the correction terms of $\epsilon'\dbc(K)$ have the form \eqref{eq:niwu} for some sequence $V_j$. This is the test
used for Montesinos knots in \cite{Lee}; its sign conventions are calibrated on the lens spaces
$L(p,2)=S^3_{p/2}(U)$. If no choice survives, then $u(K)\ne1$; this happens for \nObstrD{} knots. If some choice
survives and $\Gamma$ has one bad vertex, we compute the ranks of $HF_{\mathrm{red}}(Y_\Gamma,\mathfrak t)$ from the
graded root of $\Gamma$ and test \eqref{eq:cone} for the surviving choices; this excludes $u(K)=1$ for \nObstrHF{}
further knots. For example, for $K=\kn{13n}{2066}$ we have $D=7$, and $\Gamma$ has one bad vertex. The correction
terms of $\pm\dbc(K)$ have the form \eqref{eq:niwu} for exactly two choices, both with $V_j=0$ for all $j$, so that
$T_i=0$. Under both, $HF_{\mathrm{red}}$ has rank one at the labels $3$ and $5$ and vanishes at the other labels, so
\eqref{eq:cone} would give $c_1+c_2=0$, $c_2+c_3=1$ and $c_3+c_4=0$, which is impossible. Since $u(K)\le2$ is known for these knots \cite{Lee}, they have $u(K)=2$, which proves
Theorem~\ref{thm:main}\eqref{it:plumbing} for the \nNewUtwo{} knots of Table~\ref{tab:new-u2}. For the other knots of
class~(3) with $u(K)\in[1,2]$, some identification survives.

\subsection{Unknotting number two}\label{sec:u2}
Suppose that $u(K)\in[2,3]$ or $[2,4]$, and replace $K$ by its mirror if necessary so that $\sigma(K)\ge0$. By
Section~\ref{sec:bg-dbc}, $u(K)=2$ requires the pattern $(0,2)$ if $\sigma(K)=4$, and one of the patterns $(1,1)$ and
$(0,2)$ if $\sigma(K)=2$; if $\sigma(K)=0$ we have no test, which is the case for \nNoTest{} knots. For the patterns
with $n=\sigma(K)/2$ we apply Owens's test. We enumerate the matrices $Q$ of Section~\ref{sec:bg-dbc} in the reduced
form $0\le a<m_1\le m_2$ of \cite{Lee}, compute the minima of $\widetilde Q$ exactly in every class, and compare
them with the correction terms of $Y_\Gamma$ and of $-Y_\Gamma$ under every isomorphism of the groups. A pattern is
regarded as excluded only if no matrix fits in either orientation, so the conclusion does not depend on $\epsilon$.
The pattern $(0,2)$ for $\sigma(K)=2$ is excluded when $d_K=2$ and $\varepsilon_K=-1$.

Both patterns are excluded for \kn{13n}{2933} and \kn{13n}{2956}. These knots have $|\sigma(K)|=2$, and Regina gives
the same name, $\mathrm{SFS}\,[S^2:(3,1)(3,1)(3,2)(5,-3)]$, to their double branched covers, with
$H_1\cong\bbz/3\oplus\bbz/33$; Owens's test therefore runs over the isomorphisms of a non-cyclic group. For both
knots $d_K=2$ and $\varepsilon_K=-1$, which excludes $(0,2)$, and no matrix fits the pattern $(1,1)$ in either
orientation. Since $u(K)\le3$ is known \cite{Lee}, $u(K)=3$, which proves Theorem~\ref{thm:main}\eqref{it:plumbing}
for Table~\ref{tab:new-u3}.

The test also excludes $u(K)=2$ for two knots of Table~\ref{tab:u3-66}, \kn{13n}{3004} and \kn{13n}{3005}. They have
$|\sigma(K)|=4$, so only the pattern $(0,2)$ is possible, and no matrix fits it in either orientation; here
$H_1(\dbc(K))$ is $\bbz/3\oplus\bbz/15$ and $\bbz/3\oplus\bbz/51$, respectively. Together with $u(K)\le3$, this gives a
second proof of $u(K)=3$ for these two knots. For the other knots of class~(3) with $\sigma(K)\ne0$, at least one
pattern survives.

\subsection{Verification}\label{sec:verify}
We checked the classification of Section~\ref{sec:geometry} and the results of Sections~\ref{sec:u1}
and~\ref{sec:u2} in four ways: by a second identification of the normal forms, by the Casson--Walker invariant, by
separately written implementations of the correction-term obstructions, and on control knots.

Each normal form $\Gamma$ obtained from Regina's data is identified with $\pm\dbc(K)$ in two ways
(Section~\ref{sec:bg-plumbing}): by the Regina name that produced it, and by triangulating $Y_\Gamma$ in SnapPy as
surgery on the framed link of Figure~\ref{fig:plumbing}(a) and recognising the triangulation with Regina. The second
identification uses neither the triangulation of $\dbc(K)$ nor the conventions relating plumbing graphs to Regina's
names. It gives the same manifold for both normal forms of each of the \nNoDefRec{} knots of class~(2) whose cover
Regina recognises, and for the negative-definite normal form of each of the \nNegDef{} knots of class~(3). Of the normal
forms of the \nIdentified{} knots of Tables~\ref{tab:new-u2} and~\ref{tab:new-u3}, \nBadOne{} have one bad vertex and
\nBadTwo{} have two; those of \kn{13n}{3004} and \kn{13n}{3005} have one each.

The normal forms of the other \nGraphFour{} knots of class~(2) come from Lemma~\ref{lem:algebraic}. Their algebraic
diagrams are obtained from the KnotInfo diagrams by Reidemeister moves, and SnapPy also finds their exteriors isometric
to those of the KnotInfo diagrams. For each of these knots, a second search with a different random seed gives a second
algebraic diagram, which differs from the first in its number of crossings or in the sizes of its faces, with the same
two normal forms, and $\operatorname{coker}Q_\Gamma$ is cyclic of order $\det K$ for both. We also applied the
construction of Lemma~\ref{lem:algebraic} to the KnotInfo diagrams that are algebraic; there are \nAlgCheck{} of
them among the knots with at most thirteen crossings, $|\det Q_{\Gamma_D}|=\det K$ for all of them, and every
reduced tree is in normal form. For the \nAlgTrees{} of these knots that lie in classes~(2) and~(3), the construction gives the normal forms
identified by triangulation: both of them in class~(2), and the negative-definite one in class~(3). These knots include
\nAlgThm{} of the \nAdopted{} knots of Tables~\ref{tab:new-u2} and~\ref{tab:new-u3}.

The sign $\epsilon$ is fixed by comparing the linking pairing of $\operatorname{coker}Q_\Gamma$ with that of
$\dbc(K)$, unless the pairing is isometric to its negative. Since the tests of Sections~\ref{sec:u1}
and~\ref{sec:u2} are applied to both orientations, their conclusions do not depend on $\epsilon$; the sign is used
only in the following check. As a check that uses no Floer homology, we compare
Casson--Walker invariants $\lambda$, normalised so that $\lambda=-1$ on the boundary of the negative-definite $E_8$
plumbing, as in \cite{Lee}. In this normalisation Mullins's formula \cite{Mullins} gives
\[
    R(K):=-\frac{D\sigma(K)}4+\frac{D\,V_K'(-1)}{6\,V_K(-1)}=-2D\,\lambda(\dbc(K));
\]
see \cite{GreeneWatson}, whose normalisation is twice ours. The formula for plumbed manifolds in
\cite{Nemethi} gives
\[
    L(\Gamma):=\frac{|\det Q_\Gamma|}{12}\Bigl(\sum_ve_v+3|V|+\sum_v(2-\delta_v)\,(Q_\Gamma^{-1})_{vv}\Bigr)
    =-2|H_1(Y_\Gamma)|\,\lambda(Y_\Gamma);
\]
for example, $L(\Gamma)=2$ for the $E_8$ plumbing and $L(\Gamma)=-1/6$ for a single vertex of weight $-3$, in
agreement with $-2|H_1|\lambda=\sum_{\mathfrak t}d(Y_\Gamma,\mathfrak t)$ for these $L$-spaces \cite{Lee}.
Since $\lambda(-Y)=-\lambda(Y)$, the identification forces $\epsilon L(\Gamma)=R(K)$. This identity holds exactly for
all \nAdopted{} knots of Tables~\ref{tab:new-u2} and~\ref{tab:new-u3} and for \kn{13n}{3004} and \kn{13n}{3005}, with the sign given by the linking pairing whenever the pairing determines one.

The correction-term obstructions were recomputed by separately written programs, which share with the first ones only
the correction terms of $Y_\Gamma$. The second implementation of the test of Section~\ref{sec:u1} allows every affine
identification, not only those compatible with the linking pairing, and leaves no choice for any of the \nObstrD{} knots excluded by correction
terms. For the two knots of Table~\ref{tab:new-u3} and for \kn{13n}{3004} and \kn{13n}{3005}, the second
implementation of Owens's test enumerates a larger, unreduced family of matrices $Q$, computes the minima of
$\widetilde Q$ with a separate exact routine, and runs over all isomorphisms of the groups; no matrix fits in either
orientation. The invariants $\varepsilon_K$ and $d_K$ were recomputed from $V_K(e^{i\pi/3})$ in exact
cyclotomic arithmetic. The exclusions by reduced Floer homology were not recomputed separately; for each of the
\nObstrHF{} knots, the records list the ranks of $HF_{\mathrm{red}}(Y_\Gamma)$, every identification that survives the
correction-term test, and the condition of \eqref{eq:cone} that fails, as in the example of Section~\ref{sec:u1}.

As controls, we applied Owens's test to the \nCtrl{} Montesinos knots with at most twelve crossings, $|\sigma(K)|\in\{2,4\}$ and
$u(K)\in\{2,3\}$, using the plumbings given by their Montesinos notation. It obstructs none of the \nCtrlFourTwo{}
knots with $|\sigma(K)|=4$ and $u(K)=2$. Of the \nCtrlFourThree{} knots with $|\sigma(K)|=4$ and $u(K)=3$, it
obstructs \nCtrlFourThreeBoth{} in both orientations and \nCtrlFourThreeExp{} more in the expected orientation only.
Among the \nCtrlAlt{} alternating knots it obstructs \nCtrlAltBoth{}, all of which are also obstructed by the
Goeritz-form test of \cite{Lee}; that test obstructs \nCtrlAltGoeritz{}.

\section{Knots with \texorpdfstring{$u=u_a=3$}{u = ua = 3} and special alternating knots}\label{sec:66}

This section proves Theorem~\ref{thm:main}\eqref{it:rank2} and~\eqref{it:u5}. Neither argument needs a plumbing. For
the knots of Table~\ref{tab:u3-66} we use the linking pairing of $\dbc(K)$, and for the special alternating knots
the theorem of McCoy and Park, which reduces the lower bound to a finite check of one diagram.

\subsection{Rank-two linking pairings}\label{sec:rank2}

The knots of Table~\ref{tab:u3-66} have $u(K)\in[2,3]$ by \cite{Lee}. If $K$ can be turned into a knot with
Alexander polynomial one by at most one crossing change, then Theorem~\ref{thm:bf} presents $2\,l(K)$ by a matrix of
size at most one, so $H_1(\dbc(K))$ is cyclic. Table~\ref{tab:u3-66} shows that it is not, so $u_a(K)\ge2$. To exclude
two changes, we use the following consequence of Theorem~\ref{thm:bf}.

\begin{lemma}\label{lem:rank2}
Suppose that $K$ can be turned into a knot with Alexander polynomial one by two crossing changes. Then there are
$e_1,e_2\in\{\pm1\}$ and an integral symmetric matrix $C=\left(\begin{smallmatrix}a&c\\c&b\end{smallmatrix}\right)$
with the following properties.
\begin{enumerate}
\item $|\det C|=\det K$ and $\operatorname{sign}C=\sigma(K)+e_1+e_2$.
\item There is $P\in GL(2,\bbz)$ such that $PCP^T$ has diagonal entries congruent to $e_1$ and $e_2$ modulo four and
an even off-diagonal entry.
\item There are $v_1,v_2\in H_1(\dbc(K))$ with $2\,l(K)(v_i,v_j)\equiv(C^{-1})_{ij}\pmod\bbz$ for $i,j\in\{1,2\}$.
\item Either $0<|a|\le|b|\le\det K$ and $0\le c\le|a|/2$, or $b=0$, $\det C=-c^2$ and $-c\le a\le c$.
\end{enumerate}
\end{lemma}

\begin{proof}
Let $A(t)$ be the matrix of Theorem~\ref{thm:bf}, with $A(1)=\operatorname{diag}(e_1,e_2)$, and let $C'=A(-1)$.
Theorem~\ref{thm:bf} gives (1) and $l(C')\cong2\,l(K)$, and the remark after it gives (2) for $C'$ with $P=I$. By
\cite[Lem.~5.2]{BF}, $C'$ is congruent to a matrix $C$ as in (4). In the second case, that lemma states the
condition as $c^2=d$, where $d$ is the determinant; since
$\det\left(\begin{smallmatrix}a&c\\c&0\end{smallmatrix}\right)=-c^2$, it reads $c^2=-\det C$. Congruence
preserves (1), preserves (2) after changing $P$, and preserves the isometry class of the linking pairing, and an
isometry $l(C)\cong2\,l(K)$ gives elements $v_1,v_2$ as in (3) \cite{BF}.
\end{proof}

The conditions of Lemma~\ref{lem:rank2} can be checked by a finite computation. For both signs of $\det C$ we list the
matrices of (4). For each of them and each pair $(e_1,e_2)$ we test (1), and we test (2) over the $96$ elements of
$GL(2,\bbz/4)$. This is exactly condition (2), since the residues in (2) depend only on $P$ modulo four and every
element of $GL(2,\bbz/4)$ lifts to $GL(2,\bbz)$: as $\bbz/4$ is a local ring with units $\pm1$, the group
$GL(2,\bbz/4)$ is generated by elementary matrices and $\operatorname{diag}(-1,1)$. We test (3) by listing the elements of
$H_1(\dbc(K))=\operatorname{coker}(V+V^T)$ with the pairing $l(V+V^T)$. We do not require $v_1$ and $v_2$ to generate
$H_1(\dbc(K))$, which can only weaken the test.

For every knot of Table~\ref{tab:u3-66}, no matrix satisfies Lemma~\ref{lem:rank2}, so $u_a(K)\ge3$. For
\nUthreeNoCond{} of these knots, including all the alternating ones, no matrix with $|\det C|=\det K$ satisfies (3)
at all. The remaining \nUthreeSigned{} knots have $|\sigma(K)|=4$, so (1) forces $e_1=e_2=-\sigma(K)/4$, and no matrix
satisfies (2) and (3) with these signs.

Conversely, each knot of Table~\ref{tab:u3-66} can be unknotted by changing three crossings of its minimal diagram in
KnotInfo. For each of these sets of three crossings, SnapPy simplifies the fundamental group of the complement of the
changed knot to $\bbz$, so the changed knot is the unknot by Dehn's lemma \cite{Papakyriakopoulos}; in agreement with
this, its knot Floer homology $\widehat{HFK}$, computed with $\mathbb F_2$ coefficients, has total rank one. Hence
$u_a(K)\le u(K)\le3$, and therefore $u(K)=u_a(K)=3$. This proves Theorem~\ref{thm:main}\eqref{it:rank2}.

Borodzik and Friedl used this obstruction, with definiteness conditions when $|\sigma(K)|=4$, to show $u(K)\ge3$ for
\nBFPubCtrl{} knots with at most twelve crossings \cite{BF}. We applied Lemma~\ref{lem:rank2}
to the \nBFTested{} knots with at most thirteen crossings whose unknotting number is recorded in KnotInfo as one of
$[1,3]$, $[1,4]$, $[2,3]$, $[2,4]$ and $[2,5]$. It shows for \nBFObstructed{} of them that two crossing changes cannot
produce a knot with Alexander polynomial one: the knots of Table~\ref{tab:u3-66}, and \nBFKnownThree{} knots for which
$u(K)\ge3$ was known \cite{Lee}. As controls, the same implementation excludes none of the \nBFNegCtrl{} knots with
$u(K)\le2$ by \cite{Lee} or with $u_a(K)\le2$ in KnotInfo, and it excludes all
\nBFPubCtrl{} knots of Borodzik and Friedl and their mirrors. The computation for Table~\ref{tab:u3-66} was repeated with Seifert matrices
computed from the PD codes, with the same result. For the alternating knots it was also repeated, for the knots and
their mirrors, with a second implementation written independently, which uses integer Hermite normal forms and exact
rational arithmetic.

Among the knots of Table~\ref{tab:u3-66}, \nUthreeSigmaZero{} have $\sigma(K)=0$. For them the correction-term
arguments of Section~\ref{sec:bg-dbc} do not apply to the pattern $(1,1)$, so even when $\dbc(K)$ bounds a
negative-definite plumbing, Section~\ref{sec:new} gives no test of $u(K)=2$. Lemma~\ref{lem:rank2} treats all sign
patterns at once. It also does not need a plumbing: by Section~\ref{sec:geometry}, for \nThmCNoPlumb{} of the
\nUthreeNonalt{} non-alternating knots of Table~\ref{tab:u3-66} neither $\dbc(K)$ nor $-\dbc(K)$ bounds a
negative-definite plumbing tree, and the test of Section~\ref{sec:u2} applies only to \kn{13n}{3004} and
\kn{13n}{3005}, where it also excludes $u(K)=2$.

\subsection{Special alternating knots with \texorpdfstring{$u=5$}{u = 5}}\label{sec:u5}

Let $K$ be one of \kn{13a}{4770}, \kn{13a}{4787} and \kn{13a}{4801}. These knots have $\sigma(K)=-8$, and
$u(K)\in[4,5]$ by \cite{Lee}. The reduced alternating diagram of $K$ in KnotInfo has thirteen crossings, all of the
same sign. For such a diagram one of the two checkerboard surfaces is a Seifert surface \cite{McCoyPark}, so $K$
is special alternating, and Theorem~\ref{thm:mccoy-park} shows that $u(K)=4$ or $c_4(K)=4$ only if $K$ can be
unknotted by changing four crossings of this diagram.

We checked all sets of at most four of the thirteen crossings. For each of them, either the changed diagram has
determinant different from one, or its knot Floer homology has total rank greater than one; in neither case is it the
unknot. Hence $u(K)\ne4$ and $c_4(K)\ne4$. Since $u(K)\ge c_4(K)\ge|\sigma(K)|/2=4$ \cite{McCoyPark}, it follows
that $u(K)\ge5$ and $c_4(K)\ge5$. On the other hand, changing a suitable set of five crossings of the diagram gives the unknot: SnapPy simplifies the fundamental group of the
complement of the changed knot to $\bbz$ \cite{Papakyriakopoulos}, and its knot Floer homology has total rank one. Therefore
$u(K)=5$, and since $c_4(K)\le u(K)$, also $c_4(K)=5$. This proves Theorem~\ref{thm:main}\eqref{it:u5}. McCoy and Park
determined unknotting numbers of special alternating knots with eleven and twelve crossings in the same way
\cite{McCoyPark}; the three knots above have thirteen.

\section{A further explicit counterexample to the Bernhard--Jablan conjecture}\label{sec:13n1587}

We use the notation $N_{\min}(K)$ and $u^s_{\mathrm{BJ}}(K)$ of Section~\ref{sec:intro}. Brittenham and Hermiller
showed that the Bernhard--Jablan conjecture fails for at least one of \kn{12n}{288}, \kn{12n}{491}, \kn{12n}{501} and \kn{13n}{3370} \cite{BH1705}, and the computations of
\cite{Lee} show that it fails for \kn{13n}{3370}. For $K=\kn{13n}{1587}$, Brittenham and Hermiller proved $u(K)\le2$
and showed that the corresponding equality for the weak Bernhard--Jablan number, in which every change of an
unknotting sequence is made in a minimal diagram of the current knot, fails \cite{BH1705}. The value $u(K)=2$
was determined in \cite{Lee}, but the equality $u(K)=u^s_{\mathrm{BJ}}(K)$ itself remained open. Theorem~\ref{thm:bj}
shows that $u^s_{\mathrm{BJ}}(K)=3>2=u(K)$, so \kn{13n}{1587} is a further explicit counterexample to the
conjecture, besides \kn{13n}{3370}.

\begin{table}[htbp]
\centering
\caption{The knots obtained by changing one crossing in a minimal diagram of \kn{13n}{1587}, and the reason why $u\ge2$. For each knot, \emph{Changes} is the number of crossing changes that give it, among the \nNbrChanges{} changes of one crossing in the \nMinDiag{} minimal diagrams obtained from reference diagrams in Section~\ref{sec:13n1587}. Knots are named up to mirror image.}\label{tab:nbr}
\begin{tabular}{lrrll}
\toprule
Knot & Changes & $|\sigma|$ & $H_1(\dbc(K))$ & $u\ge2$ by\\
\midrule
$10_{38}$ & 29 & 2 & $\bbz/59$ & no unknotting crossing \cite{McCoy}\\
\kn{11n}{99} & 14 & 2 & $\bbz/39$ & \cite{GordonLuecke06}\\
$9_{46}$ & 14 & 0 & $\bbz/3 \oplus \bbz/3$ & $H_1$ needs two generators \cite{Wendt}\\
$5_2\#3_1$ & 12 & 4 & $\bbz/21$ & $|\sigma|=4$\\
\kn{12n}{155} & 8 & 4 & $\bbz/89$ & $|\sigma|=4$\\
\kn{12n}{563} & 8 & 2 & $\bbz/71$ & \cite{Lee}\\
\kn{11n}{66} & 5 & 2 & $\bbz/75$ & \cite{GordonLuecke06}\\
\kn{12a}{237} & 1 & 4 & $\bbz/161$ & $|\sigma|=4$\\
\midrule
total & 91 & & & \\
\bottomrule
\end{tabular}
\end{table}

\begin{proof}[Proof of Theorem~\ref{thm:bj}]
We first show that $u(K)=2$. The double branched cover of $K$ is an $L$-space, and its correction terms, computed
with Greene's spanning tree model, exclude $u(K)=1$ \cite{Lee}. For the upper bound, Brittenham and Hermiller
give braid representatives of $K$ and of $10_{113}$ that differ in one crossing \cite{BH1705}, and $10_{113}$
has an unknotting crossing. We also found a diagram of $K$ with fifteen crossings in which changing two crossings gives
the unknot: the changed diagram simplifies to a diagram without crossings, SnapPy simplifies the fundamental group of
its complement to $\bbz$, and its knot Floer homology has total rank one. In both arguments the diagrams were identified with the
named knots by combinatorial isomorphisms between triangulations of their complements, found with SnapPy's canonical
retriangulation; this suffices by the knot complement theorem \cite{GordonLuecke89}. Hence $u(K)=2$.

Next we determine $N_{\min}(K)$. By \cite[Prop.~2.3]{BH1705}, $N_{\min}(K)$ consists of the knots obtained by changing one
crossing in those diagrams of $K$ that arise from reference diagrams by crossing changes, where the reference diagrams
are one reduced alternating diagram for each of the \nRefAlt{} alternating knots with thirteen crossings. Since
changing all crossings of a diagram gives a diagram of the mirror image, we consider only the crossing changes that
leave a fixed crossing of the reference diagram unchanged; up to mirror image, these give all the diagrams in question.
We computed the determinant of each of them exactly; \nDetCand{} have determinant $\det K=75$. Of these, \nExclCross{}
simplify to diagrams with fewer than thirteen crossings, \nExclAlex{} have a different Alexander polynomial, and
\nExclCover{} have the same Alexander polynomial and knot Floer homology as $K$, up to mirror image, but are
distinguished from $K$ by the multiset of first homology groups of the connected $d$-fold covers of their complements,
for some $d\le4$. The remaining \nMinDiag{} diagrams represent $K$ or its mirror image. They need not be all the minimal
diagrams of $K$, but by the description of $N_{\min}(K)$ above they suffice to determine it: up to mirror image,
$N_{\min}(K)$ is the set of knots obtained by changing one crossing in one of them.

Finally, changing one crossing in each of these \nMinDiag{} diagrams gives \nNbrChanges{} diagrams, which represent
the \nNbrTypes{} knots of Table~\ref{tab:nbr} up to mirror image; since the unknotting number is invariant under mirror
image, it suffices to bound the unknotting numbers of these knots. For $9_{46}$, $10_{38}$, \kn{12a}{237}, \kn{12n}{155}
and the composite knot, the bound $u\ge2$ follows from the diagram itself: from $H_1$ of the double branched cover
\cite{Wendt}, from the signature, or, for $10_{38}$, because the diagram simplifies to a reduced alternating diagram
with ten crossings in which no crossing change gives the unknot, so that McCoy's theorem applies \cite{McCoy}.
For the composite knot, whose summands $5_2$ and $3_1$ have signatures of the same sign, $u\ge2$ also follows from
Scharlemann's theorem that knots with unknotting number one are prime \cite{Scharlemann}.
The other three knots were identified by isomorphisms of triangulations found independently with SnapPy and with
Regina. For \kn{11n}{66} and \kn{11n}{99}, Gordon and Luecke showed $u\ge2$ \cite{GordonLuecke06}; we also
excluded $u=1$ for them with correction terms from Greene's model \cite{Greene} and the Ni--Wu formula, as in
\cite{Lee}. For \kn{12n}{563}, $u=2$ by \cite{Lee}. Hence every knot in $N_{\min}(K)$ has unknotting
number at least two, and $u^s_{\mathrm{BJ}}(K)=3$.
\end{proof}

\begin{remark}
The proof uses the lower bound $u(\kn{12n}{563})\ge2$ of \cite{Lee}, but the existence of a counterexample among
\kn{13n}{1587} and \kn{12n}{563} does not depend on it. The same computation, with the \nTwelveRef{} alternating
knots with twelve crossings as references, gives \nTwelveDiag{} diagrams of \kn{12n}{563} up to mirror image, and none
of their \nTwelveChanges{} crossing changes gives the unknot, since every changed diagram has determinant different
from one. Hence, if $u(\kn{12n}{563})=1$, then \kn{12n}{563} itself is a counterexample; otherwise the proof above
shows that \kn{13n}{1587} is one.
\end{remark}

\begin{samepage}
\subsection*{Data and code availability}
The code and the computational records of this paper are available in the following repository:
\begin{center}
\coderepo
\end{center}
It contains the scripts that recognise the double branched covers, compute their Neumann normal forms and evaluate the obstructions of
Section~\ref{sec:new}, the rank-two linking pairing test of Section~\ref{sec:rank2}, the checks of
Sections~\ref{sec:u5} and~\ref{sec:13n1587}, and the script that generates the tables and counts of this paper from
the deposited records.

Parts of the code were written with the assistance of Claude Opus 5.5 (Anthropic). The author takes full
responsibility for the code, the mathematical content and the results of this paper.
\end{samepage}

\subsection*{Acknowledgements}
The author is grateful to Pavel Putrov for helpful discussions and valuable comments on an earlier draft of this
paper.

This work is supported by the Institute for Basic Science (IBS) under Project No. IBS-R003-D1.

\appendix
\section{Tables}
\begin{longtable}{lrrllr}
\caption{New exact values $u(K)=2$ (previous range $[1,2]$). $|V|$ is the number of vertices of the Neumann normal form.}\label{tab:new-u2}\\
\toprule
Knot & $\det$ & $\sigma$ & $\dbc(K)$ (Regina) & Obstruction & $|V|$\\
\midrule
\endfirsthead
\multicolumn{6}{@{}l@{}}{\emph{New exact values $u(K)=2$, continued from the previous page.}}\\[2pt]
\toprule
Knot & $\det$ & $\sigma$ & $\dbc(K)$ (Regina) & Obstruction & $|V|$\\
\midrule
\endhead
\bottomrule
\multicolumn{6}{@{}r@{}}{\emph{continued on the next page}}\\
\endfoot
\bottomrule
\endlastfoot
\kn{12n}{51} & 9 & 0 & {\footnotesize [D: (2,1) (5,3)] $\cup$ [D: (3,1) (3,2)], m = [ 0,1 | 1,0 ]} & $d$ & 8\\
\kn{12n}{184} & 43 & 2 & {\footnotesize [D: (2,1) (3,2)] $\cup$ [D: (3,2) (5,3)], m = [ 0,1 | 1,0 ]} & $d$ & 7\\
\kn{13n}{125} & 59 & -2 & {\footnotesize [D: (2,1) (5,3)] $\cup$ [D: (3,2) (5,3)], m = [ 0,1 | 1,0 ]} & $d$ & 8\\
\kn{13n}{315} & 53 & 0 & {\footnotesize [D: (2,1) (3,2)] $\cup$ [D: (2,1) (9,7)], m = [ 0,1 | 1,0 ]} & $d$ & 7\\
\kn{13n}{560} & 27 & -2 & {\footnotesize [S2: (3,1) (5,2) (12,-7)]} & $\mathrm{HF}_{\mathrm{red}}$ & 8\\
\kn{13n}{567} & 31 & -2 & {\footnotesize [D: (2,1) (3,1)] $\cup$ [D: (3,1) (5,2)], m = [ 1,1 | 1,0 ]} & $d$ & 9\\
\kn{13n}{572} & 9 & 0 & {\footnotesize [S2: (3,1) (5,2) (6,-5)]} & $\mathrm{HF}_{\mathrm{red}}$ & 5\\
\kn{13n}{574} & 39 & -2 & {\footnotesize [S2: (3,1) (5,2) (9,-4)]} & $d$ & 9\\
\kn{13n}{611} & 9 & 0 & {\footnotesize [D: (2,1) (3,2)] $\cup$ [D: (3,1) (3,2)], m = [ -1,2 | 0,1 ]} & $d$ & 8\\
\kn{13n}{612} & 9 & 0 & {\footnotesize [D: (2,1) (3,2)] $\cup$ [D: (3,1) (3,2)], m = [ -1,2 | 0,1 ]} & $d$ & 8\\
\kn{13n}{751} & 25 & 0 & {\footnotesize [D: (2,1) (3,2)] $\cup$ [D: (5,2) (5,3)], m = [ 0,1 | 1,0 ]} & $d$ & 8\\
\kn{13n}{761} & 69 & 0 & {\footnotesize [D: (2,1) (3,2)] $\cup$ [D: (4,3) (5,3)], m = [ 0,1 | 1,0 ]} & $d$ & 7\\
\kn{13n}{860} & 77 & 0 & {\footnotesize [D: (2,1) (3,2)] $\cup$ [D: (3,2) (7,5)], m = [ 0,1 | 1,0 ]} & $d$ & 7\\
\kn{13n}{931} & 27 & 2 & {\footnotesize [S2: (2,1) (3,1) (3,2) (7,-12)]} & $d$ & 7\\
\kn{13n}{968} & 69 & 0 & {\footnotesize [S2: (2,1) (3,1) (3,1) (7,-12)]} & $d$ & 6\\
\kn{13n}{1002} & 43 & -2 & {\footnotesize [S2: (2,1) (3,1) (29,-17)]} & $d$ & 8\\
\kn{13n}{1021} & 7 & -2 & {\footnotesize [S2: (2,1) (3,1) (17,-13)]} & $d$ & 8\\
\kn{13n}{1059} & 27 & 2 & {\footnotesize [S2: (2,1) (3,1) (3,2) (7,-12)]} & $d$ & 7\\
\kn{13n}{1062} & 27 & -2 & {\footnotesize [S2: (2,1) (3,1) (3,2) (7,-12)]} & $d$ & 7\\
\kn{13n}{1083} & 69 & 0 & {\footnotesize [S2: (2,1) (3,1) (3,1) (7,-12)]} & $d$ & 6\\
\kn{13n}{1133} & 15 & -2 & {\footnotesize [S2: (4,1) (5,1) (5,-3)]} & $\mathrm{HF}_{\mathrm{red}}$ & 5\\
\kn{13n}{1226} & 7 & 2 & {\footnotesize [D: (2,1) (3,2)] $\cup$ [D: (3,1) (7,4)], m = [ 0,1 | 1,0 ]} & $d$ & 9\\
\kn{13n}{1231} & 29 & 0 & {\footnotesize [D: (2,1) (3,2)] $\cup$ [D: (3,1) (5,4)], m = [ 0,1 | 1,0 ]} & $d$ & 7\\
\kn{13n}{1236} & 35 & -2 & {\footnotesize [D: (2,1) (3,2)] $\cup$ [D: (3,2) (7,3)], m = [ 0,1 | 1,0 ]} & $d$ & 7\\
\kn{13n}{1286} & 33 & 0 & {\footnotesize [S2: (2,1) (3,1) (3,1) (7,-10)]} & $d$ & 6\\
\kn{13n}{1287} & 9 & 0 & {\footnotesize [S2: (2,1) (3,1) (3,2) (7,-11)]} & $d$ & 8\\
\kn{13n}{1313} & 9 & 0 & {\footnotesize [S2: (2,1) (3,1) (3,2) (7,-11)]} & $d$ & 8\\
\kn{13n}{1314} & 33 & 0 & {\footnotesize [S2: (2,1) (3,1) (3,1) (7,-10)]} & $d$ & 6\\
\kn{13n}{1316} & 9 & 0 & {\footnotesize [S2: (2,1) (3,1) (3,2) (7,-11)]} & $d$ & 8\\
\kn{13n}{1354} & 35 & -2 & {\footnotesize [S2: (5,2) (5,2) (8,-5)]} & $d$ & 8\\
\kn{13n}{1435} & 9 & 0 & {\footnotesize [S2: (2,1) (7,2) (11,-8)]} & $d$ & 9\\
\kn{13n}{1552} & 93 & 0 & {\footnotesize [D: (3,2) (3,2)] $\cup$ [D: (3,2) (5,3)], m = [ 0,1 | 1,0 ]} & $d$ & 7\\
\kn{13n}{1588} & 23 & 2 & {\footnotesize [S2: (2,1) (7,2) (11,-7)]} & $d$ & 9\\
\kn{13n}{2066} & 7 & 2 & {\footnotesize [S2: (2,1) (5,1) (9,-7)]} & $\mathrm{HF}_{\mathrm{red}}$ & 5\\
\kn{13n}{2276} & 19 & 2 & {\footnotesize [D: (2,1) (3,1)] $\cup$ [D: (3,2) (7,5)], m = [ 0,1 | 1,0 ]} & $d$ & 8\\
\kn{13n}{2296} & 21 & 0 & {\footnotesize [S2: (3,1) (4,1) (9,-7)]} & $\mathrm{HF}_{\mathrm{red}}$ & 5\\
\kn{13n}{2559} & 47 & -2 & {\footnotesize [S2: (2,1) (7,3) (9,-5)]} & $d$ & 6\\
\kn{13n}{2638} & 31 & -2 & {\footnotesize [S2: (3,1) (4,1) (13,-5)]} & $d$ & 9\\
\kn{13n}{2950} & 9 & 0 & {\footnotesize [D: (3,1) (3,2)] $\cup$ [D: (3,1) (4,3)], m = [ 0,1 | 1,0 ]} & $d$ & 8\\
\kn{13n}{3163} & 19 & 2 & {\footnotesize [S2: (3,1) (5,1) (7,-5)]} & $\mathrm{HF}_{\mathrm{red}}$ & 5\\
\kn{13n}{3547} & 11 & -2 & {\footnotesize [S2: (3,1) (5,1) (7,-3)]} & $d$ & 10\\
\end{longtable}
\begin{longtable}{lrrllr}
\caption{New exact values $u(K)=3$ (previous range $[2,3]$). Owens's theorem with normal-form correction terms excludes two crossing changes of opposite signs, and Traczyk's criterion excludes two changes of the same sign (Section~\ref{sec:u2}).}\label{tab:new-u3}\\
\toprule
Knot & $\det$ & $\sigma$ & $\dbc(K)$ (Regina) & Obstruction & $|V|$\\
\midrule
\endfirsthead
\multicolumn{6}{@{}l@{}}{\emph{New exact values $u(K)=3$, continued from the previous page.}}\\[2pt]
\toprule
Knot & $\det$ & $\sigma$ & $\dbc(K)$ (Regina) & Obstruction & $|V|$\\
\midrule
\endhead
\bottomrule
\multicolumn{6}{@{}r@{}}{\emph{continued on the next page}}\\
\endfoot
\bottomrule
\endlastfoot
\kn{13n}{2933} & 99 & 2 & {\footnotesize [S2: (3,1) (3,1) (3,2) (5,-3)]} & Owens $(1,1)$ + Traczyk & 8\\
\kn{13n}{2956} & 99 & -2 & {\footnotesize [S2: (3,1) (3,1) (3,2) (5,-3)]} & Owens $(1,1)$ + Traczyk & 8\\
\end{longtable}
\begin{table}[htbp]
\centering
\setlength{\aboverulesep}{0pt}\setlength{\belowrulesep}{0pt}\renewcommand{\arraystretch}{1.05}
\caption{Knots with $u(K)=u_a(K)=3$, listed down the columns.}\label{tab:u3-66}
\begin{tabular}{lrrl@{\quad}|@{\quad}lrrl}
\toprule
Knot & $\det$ & $\sigma$ & $H_1(\dbc(K))$ & Knot & $\det$ & $\sigma$ & $H_1(\dbc(K))$\\
\midrule
\kn{13a}{290} & 325 & 0 & $\bbz/5 \oplus \bbz/65$ & \kn{13n}{3005} & 153 & -4 & $\bbz/3 \oplus \bbz/51$\\
\kn{13a}{336} & 325 & 0 & $\bbz/5 \oplus \bbz/65$ & \kn{13n}{3099} & 125 & -4 & $\bbz/5 \oplus \bbz/25$\\
\kn{13a}{1665} & 425 & 0 & $\bbz/5 \oplus \bbz/85$ & \kn{13n}{3127} & 75 & -2 & $\bbz/5 \oplus \bbz/15$\\
\kn{13a}{1804} & 325 & 0 & $\bbz/5 \oplus \bbz/65$ & \kn{13n}{3231} & 125 & -4 & $\bbz/5 \oplus \bbz/25$\\
\kn{13a}{1829} & 375 & 2 & $\bbz/5 \oplus \bbz/75$ & \kn{13n}{3253} & 169 & -4 & $\bbz/13 \oplus \bbz/13$\\
\kn{13a}{2079} & 275 & 2 & $\bbz/5 \oplus \bbz/55$ & \kn{13n}{3329} & 125 & 4 & $\bbz/5 \oplus \bbz/25$\\
\kn{13a}{2181} & 275 & -2 & $\bbz/5 \oplus \bbz/55$ & \kn{13n}{3446} & 169 & -4 & $\bbz/13 \oplus \bbz/13$\\
\kn{13a}{2277} & 425 & 0 & $\bbz/5 \oplus \bbz/85$ & \kn{13n}{3525} & 189 & 4 & $\bbz/3 \oplus \bbz/63$\\
\kn{13a}{2381} & 375 & -2 & $\bbz/5 \oplus \bbz/75$ & \kn{13n}{3546} & 25 & -4 & $\bbz/5 \oplus \bbz/5$\\
\kn{13a}{2507} & 175 & -2 & $\bbz/5 \oplus \bbz/35$ & \kn{13n}{3570} & 169 & -4 & $\bbz/13 \oplus \bbz/13$\\
\kn{13a}{2594} & 325 & 0 & $\bbz/5 \oplus \bbz/65$ & \kn{13n}{3695} & 125 & -4 & $\bbz/5 \oplus \bbz/25$\\
\kn{13a}{2933} & 375 & -2 & $\bbz/5 \oplus \bbz/75$ & \kn{13n}{3718} & 45 & -4 & $\bbz/3 \oplus \bbz/15$\\
\kn{13a}{3118} & 475 & -2 & $\bbz/5 \oplus \bbz/95$ & \kn{13n}{3720} & 175 & 2 & $\bbz/5 \oplus \bbz/35$\\
\kn{13a}{3131} & 375 & -2 & $\bbz/5 \oplus \bbz/75$ & \kn{13n}{3902} & 75 & -2 & $\bbz/5 \oplus \bbz/15$\\
\kn{13a}{3302} & 425 & 0 & $\bbz/5 \oplus \bbz/85$ & \kn{13n}{3923} & 75 & -2 & $\bbz/5 \oplus \bbz/15$\\
\kn{13a}{3407} & 425 & 0 & $\bbz/5 \oplus \bbz/85$ & \kn{13n}{4262} & 175 & 2 & $\bbz/5 \oplus \bbz/35$\\
\kn{13a}{3581} & 325 & 0 & $\bbz/5 \oplus \bbz/65$ & \kn{13n}{4401} & 117 & -4 & $\bbz/3 \oplus \bbz/39$\\
\kn{13a}{3613} & 325 & 0 & $\bbz/5 \oplus \bbz/65$ & \kn{13n}{4674} & 125 & 0 & $\bbz/5 \oplus \bbz/25$\\
\kn{13a}{4030} & 375 & 2 & $\bbz/5 \oplus \bbz/75$ & \kn{13n}{4695} & 125 & -4 & $\bbz/5 \oplus \bbz/25$\\
\kn{13n}{714} & 125 & 0 & $\bbz/5 \oplus \bbz/25$ & \kn{13n}{4696} & 25 & 4 & $\bbz/5 \oplus \bbz/5$\\
\kn{13n}{989} & 75 & 2 & $\bbz/5 \oplus \bbz/15$ & \kn{13n}{4735} & 125 & -4 & $\bbz/5 \oplus \bbz/25$\\
\kn{13n}{1662} & 117 & 4 & $\bbz/3 \oplus \bbz/39$ & \kn{13n}{4771} & 175 & 2 & $\bbz/5 \oplus \bbz/35$\\
\kn{13n}{1992} & 75 & -2 & $\bbz/5 \oplus \bbz/15$ & \kn{13n}{4783} & 125 & 4 & $\bbz/5 \oplus \bbz/25$\\
\kn{13n}{2034} & 75 & 2 & $\bbz/5 \oplus \bbz/15$ & \kn{13n}{4800} & 117 & -4 & $\bbz/3 \oplus \bbz/39$\\
\kn{13n}{2245} & 175 & 2 & $\bbz/5 \oplus \bbz/35$ & \kn{13n}{4805} & 117 & -4 & $\bbz/3 \oplus \bbz/39$\\
\kn{13n}{2305} & 25 & 4 & $\bbz/5 \oplus \bbz/5$ & \kn{13n}{4831} & 175 & -2 & $\bbz/5 \oplus \bbz/35$\\
\kn{13n}{2526} & 25 & -4 & $\bbz/5 \oplus \bbz/5$ & \kn{13n}{4870} & 175 & -2 & $\bbz/5 \oplus \bbz/35$\\
\kn{13n}{2550} & 45 & -4 & $\bbz/3 \oplus \bbz/15$ & \kn{13n}{4903} & 169 & -4 & $\bbz/13 \oplus \bbz/13$\\
\kn{13n}{2688} & 125 & 0 & $\bbz/5 \oplus \bbz/25$ & \kn{13n}{4927} & 153 & 4 & $\bbz/3 \oplus \bbz/51$\\
\kn{13n}{2691} & 189 & -4 & $\bbz/3 \oplus \bbz/63$ & \kn{13n}{4954} & 125 & 4 & $\bbz/5 \oplus \bbz/25$\\
\kn{13n}{2791} & 75 & 2 & $\bbz/5 \oplus \bbz/15$ & \kn{13n}{5008} & 225 & -4 & $\bbz/5 \oplus \bbz/45$\\
\kn{13n}{2833} & 125 & 4 & $\bbz/5 \oplus \bbz/25$ & \kn{13n}{5076} & 225 & -4 & $\bbz/15 \oplus \bbz/15$\\
\kn{13n}{3004} & 45 & -4 & $\bbz/3 \oplus \bbz/15$ & \kn{13n}{5085} & 175 & -2 & $\bbz/5 \oplus \bbz/35$\\
\bottomrule
\end{tabular}
\end{table}
\FloatBarrier

\end{document}